\documentclass[a4paper]{amsart}
\usepackage[all]{xy}
\usepackage{amssymb,enumitem,nicefrac,mathtools}
\usepackage[british]{babel}

\usepackage[pdftitle={C*-Algebras over Topological Spaces: Filtrated K-Theory},
  pdfauthor={Ralf Meyer and Ryszard Nest},
  pdfsubject={Mathematics; MSC 18E30, 19K35, 46L80, 55U35}
]{hyperref}
\usepackage[lite]{amsrefs}
\usepackage{microtype}

\newcommand*{\MRref}[2]{ \href{http://www.ams.org/mathscinet-getitem?mr=#1}{MR \textbf{#1}}}

\renewcommand{\PrintDOI}[1]{\href{http://dx.doi.org/\detokenize{#1}}{doi: \detokenize{#1}}%
  \IfEmptyBibField{pages}{, (to appear in print)}{}}

\DeclareMathOperator{\Hom}{Hom}
\DeclareMathOperator{\Prim}{Prim}
\DeclareMathOperator{\Ext}{Ext}
\DeclareMathOperator{\Tor}{Tor}
\DeclareMathOperator{\coker}{coker}
\DeclareMathOperator{\range}{range}

\newcommand*{\Left}{\mathbb L}
\newcommand*{\KK}{\textup{KK}}
\newcommand*{\K}{\textup{K}}
\newcommand*{\Good}{\textup{Good}}
\newcommand*{\Ch}{\textup{Ch}}  % order complex
\newcommand*{\nil}{\textup{nil}}% nilpotent
\newcommand*{\sesi}{\textup{ss}}  % semi-simple

\newcommand*{\into}{\rightarrowtail}
\newcommand*{\prto}{\twoheadrightarrow}

\newcommand*{\cl}[1]{\overline{#1}} % closure

\newcommand*{\CONT}{\textup C}    % continuous functions
\newcommand*{\Repr}{\mathcal R}    % certain representing objects
\newcommand*{\Nattrafo}{\mathcal{NT}} % natural transformations on filtered K-theory
\newcommand*{\FK}{\textup{FK}}% filtrated K-theory

\newcommand*{\Cstarcat}{\mathfrak{C^*alg}}
\newcommand*{\Cstarsep}{\mathfrak{C^*sep}}
\newcommand*{\KKcat}{\mathfrak{KK}}

\newcommand*{\Cat}{\mathfrak C}% generic category, often Abelian
\newcommand*{\Tri}{\mathfrak T}% triangulated category
\newcommand*{\Gen}{\mathfrak G}% generating set
\newcommand*{\Ideal}{\mathfrak I}% ideal in a triangulated category
\newcommand*{\Null}{\mathcal N}% contractible objects
\newcommand*{\Abel}{\mathcal A}% generic Abelian category
\newcommand*{\Ab}{\mathfrak{Ab}}% category of Abelian groups
\newcommand*{\Mod}[1]{\mathfrak{Mod}(#1)}% category of modules
\newcommand*{\CMod}[1]{\mathfrak{Mod}(#1)_\textup c}% category of countable modules

\newcommand*{\op}{\textup{op}}
\newcommand*{\ID}{\textup{id}}

\newcommand*{\C}{\mathbb C}
\newcommand*{\R}{\mathbb R}
\newcommand*{\Z}{\mathbb Z}
\newcommand*{\N}{\mathbb N}
\newcommand*{\Ideals}{\mathbb I}% ideal lattice of a C*-algebra
\newcommand*{\Open}{\mathbb O}% open subsets of a top. space
\newcommand*{\Loclo}{\mathbb{LC}}% locally closed subsets of a top. space

\newcommand*{\nb}{\nobreakdash}% no break after this hyphen

\newcommand*{\Bootstrap}{\mathcal B}

\newcommand*{\Cst}{\textup C^*}
\newcommand*{\Star}{\texorpdfstring{$^*$\nb-}{*-}}

\newcommand*{\blank}{\text{\textvisiblespace}}
\newcommand*{\inOb}{\mathrel{\in\in}}

\newcommand*{\defeq}{\mathrel{\vcentcolon=}}

\newcommand*{\lad}{\vdash}

\theoremstyle{plain}
\newtheorem{theorem}{Theorem}
\newtheorem{proposition}[theorem]{Proposition}
\newtheorem{lemma}[theorem]{Lemma}
\newtheorem{corollary}[theorem]{Corollary}
\theoremstyle{definition}
\newtheorem{definition}[theorem]{Definition}
\theoremstyle{remark}
\newtheorem{remark}[theorem]{Remark}
\newtheorem{example}[theorem]{Example}

\numberwithin{equation}{section}
\numberwithin{theorem}{section}

\begin{document}

\title[\(\Cst\)-Algebras over Topological Spaces: Filtrated $\K$-Theory]{\(\Cst\)-Algebras over Topological Spaces:\\ Filtrated $\K$-Theory}

\author{Ralf Meyer}
\address{Mathematisches Institut and Courant Research Centre ``Higher Order Structures''\\
  Georg-August Universit\"at G\"ottingen\\
  Bunsenstra{\ss}e 3--5\\
  37073 G\"ottingen\\
  Germany}
\email{rameyer@uni-math.gwdg.de}

\author{Ryszard Nest}
\address{K{\o}benhavns Universitets Institut for Matematiske Fag\\
         Universitetsparken 5\\ 2100 K{\o}benhavn\\ Denmark
}
\email{rnest@math.ku.dk}


\begin{abstract}
  We define the filtrated \(\K\)\nb-theory of a \(\Cst\)\nb-algebra over a finite topological space~\(X\) and explain how to construct a spectral sequence that computes the bivariant Kasparov theory over~\(X\) in terms of filtrated \(\K\)\nb-theory.

  For finite spaces with totally ordered lattice of open subsets, this spectral sequence becomes an exact sequence as in the Universal Coefficient Theorem, with the same consequences for classification.

  We also exhibit an example where filtrated \(\K\)\nb-theory is not yet a complete invariant.  We describe two \(\Cst\)\nb-algebras over a space~\(X\) with four points that have isomorphic filtrated \(\K\)\nb-theory without being \(\KK(X)\)\nb-equivalent.  For this space~\(X\), we enrich filtrated \(\K\)\nb-theory by another \(\K\)\nb-theory functor to a complete invariant up to \(\KK(X)\)\nb-equivalence that satisfies a Universal Coefficient Theorem.
\end{abstract}

\thanks{The second author was supported by the German Research Foundation (Deutsche Forschungsgemeinschaft (DFG)) through the Institutional Strategy of the University of G\"ottingen.}
\maketitle

\section{Introduction}
\label{sec:intro}

\subsection{The UCT-problem}

One of the main problems in the theory of \(\Cst\)-algebras is the computation of the equivariant \(\KK\)-theory of \(\Cst\)-algebras endowed with some extra structure.  Here we apply the general techniques developed in \cites{Meyer-Nest:Homology_in_KK, Meyer:Homology_in_KK_II} to the case of \(\Cst\)\nb-algebras with a non-trivial ideal lattice.  The appropriate version of \(\KK\)-theory is Kirchberg's generalisation of Kasparov theory to \(\Cst\)\nb-algebras over non-Hausdorff topological spaces (see~\cite{Kirchberg:Michael}).  Our goal is to compute it in terms of more manageable \(\K\)\nb-theoretic information, generalising the usual Universal Coefficient Theorem that computes Kasparov's original theory for \(\Cst\)\nb-algebras in the bootstrap class by an exact sequence
\begin{equation}
  \label{eq:UCT simple}
  \Ext\bigl(\K_{*+1}(A),\K_*(B)\bigr) \into
  \KK_*(A,B) \prto
  \Hom\bigl(\K_*(A),\K_*(B)\bigr).
\end{equation}

The generalisation of the bootstrap class to the case of \(\Cst\)\nb-algebras with non-trivial ideal lattice was introduced and studied in~\cite{Meyer-Nest:Bootstrap}.  Let us first recall some of the notation from~\cite{Meyer-Nest:Bootstrap}.  Let~\(X\) be a (usually non-Hausdorff) topological space.  A \(\Cst\)\nb-algebra over~\(X\) is a \(\Cst\)\nb-algebra~\(A\) endowed with a continuous map \(\Prim(A)\to X\).  Let \(\Cstarcat(X)\) be the category of \(\Cst\)\nb-algebras over~\(X\); the morphisms in \(\Cstarcat(X)\) are given by \(X\)\nb-equivariant (in obvious sense) \Star{}homomorphisms.  Taking Kirchberg's \(\KK\)-groups as morphisms and the same objects, we get the category \(\KKcat(X)\).  It has a structure of a triangulated category (see~\cite{Meyer-Nest:Bootstrap}).  For finite~\(X\), the bootstrap class \(\Bootstrap(X)\) is defined as the smallest subcategory of \(\KKcat(X)\) that is closed under suspension, isomorphism, exact triangles, and direct sums and contains all objects with underlying \(\Cst\)\nb-algebra~\(C\).

General methods from homological algebra suggest to study a homology theory~\(H_*\) for \(\Cst\)\nb-algebras over~\(X\), taking values in some Abelian category~\(\Cat\).  Under some mild assumptions, the machinery developed in \cites{Meyer-Nest:Homology_in_KK, Meyer:Homology_in_KK_II} yields an Adams type spectral sequence which abuts to \(\KK(X;\blank,\blank)\), with an \(E_2\)\nb-term expressed in terms of~\(H_*\).

For classification purposes, we need, instead of a spectral sequence, a short exact sequence of the type~\eqref{eq:UCT simple}:
\begin{equation}
  \label{eq:UCT_generic}
  \Ext_\Cat\bigl(H_{*+1}(A),H_*(B)\bigr) \into
  \KK_*(X;A,B) \prto
  \Hom_\Cat\bigl(H_*(A),H_*(B)\bigr),
\end{equation}
and a precise description of the range of~\(H_*\).

In this case, given two \(\Cst\)\nb-algebras \(A\) and~\(B\) over~\(X\) that belong to the bootstrap class, an isomorphism of \(H_*(A)\) to \(H_*(B)\) lifts to a \(\KK(X)\)-equivalence between \(A\) and~\(B\).  The results of Eberhard Kirchberg then allow to lift this \(\KK(X)\)-equivalence to a \Star{}isomorphism \(A\cong B\), provided \(A\) and~\(B\) are tight, purely infinite, stable, nuclear and separable; here \emph{tightness} means that the maps \(\Prim(A)\to X\) and \(\Prim(B)\to X\) are homeomorphisms (see~\cite{Kirchberg:Michael}).  It is also shown in~\cite{Meyer-Nest:Bootstrap} that, in the case when~\(X\) is finite, any object of the bootstrap class is \(\KK(X)\)\nb-equivalent to a tight, purely infinite, stable, nuclear, separable \(\Cst\)\nb-algebra over~\(X\).

\emph{Hence the existence of an exact sequence of the form~\eqref{eq:UCT_generic} for all objects of the bootstrap class leads to a complete classification of the tight, purely infinite, stable, nuclear, separable \(\Cst\)\nb-algebras over~\(X\) in terms of their image under the functor~\(H_*\).}

\subsection{Main results}

It is relatively easy to construct filtrations on \(\KKcat\) which produce spectral sequences which converge to \(\KK\)-groups on the bootstrap category and whose \(E_2\)\nb-term involves only the \(\K\)\nb-theory of the quotients \(\K_*(A/J)\) for the ideals~\(J\) corresponding to minimal open subsets of~\(X\); an example is the filtration used in \cite{Meyer-Nest:Bootstrap}*{Section 4.1}.  However, this spectral sequence is not very useful for practical purposes, since it does not degenerate at the \(E_2\)\nb-level.  The second differential involves, in particular, the \(\K\)\nb-theory of various subquotients~\(I/J\) for the ideals \(I\subset J\subset A\) and the associated six-term exact sequences in \(\K\)\nb-theory
\begin{equation}
  \label{eq:six-term_intro 0}
  \begin{gathered}
    \xymatrix{
      \K_0\bigl(I\bigr) \ar[r]& \K_0\bigl(J\bigr) \ar[r]&
      \K_0\bigl(J/I\bigr) \ar[d]\\
      \K_1\bigl(J/I\bigr) \ar[u]& \K_1\bigl(J\bigr) \ar[l]& \K_1\bigl(I\bigr). \ar[l] }
  \end{gathered}
\end{equation}
Also higher differentials do not vanish.

To get a short exact sequence instead, we need to consider more sophisticated homology theories.  The homology theory analysed here is ``filtrated K-theory,'' which is in some sense the second approximation to this spectral sequence.  Roughly speaking, filtrated \(\K\)\nb-theory comprises the \(\K\)\nb-theory of various subquotients together with \emph{all canonical maps} between these groups.  We will make this definition precise later.  The part of it which involves the exact sequences~\eqref{eq:six-term_intro 0} appeared previously in the work of Gunnar Restorff~\cite{Restorff:Classification} for Cuntz--Krieger algebras and of Mikael R{\o}rdam~\cite{Rordam:Classification_extensions} and Alexander Bonkat~\cite{Bonkat:Thesis} for extensions of \(\Cst\)\nb-algebras. The UCT theorem in the case when the ideal structure is given by $I_1\triangleleft I_2\triangleleft A$ was obtained by Gunnar Restorff in his phd-thesis~\cite{Restorff:Thesis}, where he introduced an invariant which is a particular case of filtrated \(\K\)\nb-theory.

In this paper we prove the following

 \begin{theorem}
   The filtrated \(\K\)\nb-theory satisfies the Universal Coefficient Theorem and is a complete invariant for \(\Cst\)\nb-algebras over those finite topological spaces with a totally ordered lattice of open subsets.
 \end{theorem}

 Note that a \(\Cst\)\nb-algebra over a space of the type described in this result is essentially the same as a \(\Cst\)\nb-algebra~\(A\) together with a finite increasing chain of ideals
 \[
 \{0\} = I_0 \triangleleft I_1 \triangleleft I_2 \triangleleft I_3 \triangleleft \dotsb \triangleleft I_{n-1} \triangleleft I_n = A.
 \]

 We will also show that the spectral sequence associated to the filtrated \(\K\)\nb-theory does not collapse in general.  Let \((X,<)\) be the partially ordered set, where \(X=\{1,2,3,4\}\) with the partial order given by \(1,2,3<4\) and no further strict inequalities between \(1,2,3\).  A \(\Cst\)\nb-algebra over this space is a \(\Cst\)\nb-algebra~\(A\) together with an ideal~\(I\) and a decomposition of \(A/I\) into a direct sum of three orthogonal ideals.

\begin{theorem}
  The filtrated \(\K\)\nb-theory over \((X,<)\) does not satisfy the Universal Coefficient Theorem and is not a complete invariant.
\end{theorem}

In fact, we give an explicit example of two \(\Cst\)\nb-algebras \(A\) and~\(B\) over~\(X\) in the bootstrap class that have isomorphic filtrated \(\K\)\nb-theory but are not \(\KK(X)\)\nb-equivalent.

\emph{However, for the particular four-point space~\(X\), we still get a complete invariant and a Universal Coefficient Theorem as in~\eqref{eq:UCT_generic}, by adding another \(\K\)\nb-theory functor to filtrated \(\K\)\nb-theory.}

It is not clear how to construct such an enriched and still computable filtrated \(\K\)\nb-theory for general finite spaces.

\subsection{The general machinery}

Now we explain the general machinery behind our approach.  Let us fix a finite topological space~\(X\).  The first step is the correct definition of filtrated \(\K\)\nb-theory.  The filtrated \(\K\)\nb-theory of a \(\Cst\)\nb-algebra~\(A\) over~\(X\) comprises the \(\Z/2\)\nb-graded Abelian groups \(\K_*\bigl(A(Y)\bigr)\) for all locally closed subsets \(Y\subseteq X\) together with all natural transformations between these groups.  The main issue here is to find \emph{all} natural transformations.  These natural transformations enter in the definition of the target category of the filtrated \(\K\)\nb-theory functor and thus influence the \(\Hom\) and \(\Ext\) terms that we expect in the Universal Coefficient Theorem.

We can guess some of these natural transformations.  If~\(U\) is a relatively open subset of~\(Y\), then \(A(U)\) is an ideal in \(A(Y)\), with quotient \(A(Y)/A(U) = A(Y\setminus U)\).  This \(\Cst\)\nb-algebra extension leads to a \emph{natural} six-term exact sequence
\begin{equation}
  \label{eq:six-term_intro}
  \begin{gathered}
    \xymatrix{
      \K_0\bigl(A(U)\bigr) \ar[r]&
      \K_0\bigl(A(Y)\bigr) \ar[r]&
      \K_0\bigl(A(Y\setminus U)\bigr) \ar[d]\\
      \K_1\bigl(A(Y\setminus U)\bigr) \ar[u]&
      \K_1\bigl(A(Y)\bigr) \ar[l]&
      \K_1\bigl(A(U)\bigr). \ar[l]
    }
  \end{gathered}
\end{equation}
These exact sequences provide three types of natural transformations associated to inclusions of open subsets, restriction to closed subset, and boundary maps.

An obvious source for relations between these natural transformations are morphisms of \(\Cst\)\nb-algebra extensions: since the six-term exact sequences in~\eqref{eq:six-term_intro} are natural, each natural morphism of extensions provides some commuting diagrams, which become relations between our generators.

But do these obvious generators and relations already describe \emph{all} natural transformations?  This turns out to be the case for the spaces studied in this article---both the positive and the negative examples.  Although the authors know no counterexamples, we do not expect this to be so in general.

The starting point for our study of filtrated \(\K\)\nb-theory is that the covariant functors \(A\mapsto \K_*\bigl(A(Y)\bigr)\) are representable, that is, they are of the form \(\KK_*(X;\Repr_Y,A)\) for suitable \(\Cst\)\nb-algebras~\(\Repr_Y\) over~\(X\)---these are the representing objects.  Our construction of~\(\Repr_Y\) yields commutative \(\Cst\)\nb-algebras, consisting of \(\CONT_0\)\nb-functions on suitable locally closed subspaces of the order complex of the partial order on~\(X\).  The Yoneda Lemma tells us that natural transformations from \(\K_*\bigl(A(Y)\bigr)\) to \(\K_*\bigl(A(Z)\bigr)\) correspond to \(\KK_*(X;\Repr_Z,\Repr_Y) \cong \K_*\bigl(\Repr_Y(Z)\bigr)\).  These groups are easy enough to compute in the examples we consider, and turn out to be definable by the concrete generators and relations mentioned above.

The natural transformations acting on filtrated \(\K\)\nb-theory form a \(\Z/2\)-graded pre-additive category~\(\Nattrafo\).  A (countable) module over~\(\Nattrafo\) is, by definition, an additive functor from~\(\Nattrafo\) to the category of (countable) \(\Z/2\)-graded Abelian groups.  By construction, the filtrated \(\K\)\nb-theory of any \(\Cst\)\nb-algebra over~\(X\) is such a countable module.  Let~\(\Cat\) be the category of countable \(\Nattrafo\)\nb-modules.  This is an Abelian category, and filtrated \(\K\)\nb-theory is a stable homological functor~\(\FK\) from the Kasparov category \(\KKcat(X)\) of \(\Cst\)\nb-algebras over~\(X\) to~\(\Cat\).

It is easy to check that the functor \(\FK\colon \KKcat(X)\to\Cat\) is \emph{universal} in the notation of~\cite{Meyer-Nest:Homology_in_KK}.  General results on homological ideals in triangulated categories now produce a cohomological spectral sequence that converges towards \(\KK_*(X;A,B)\) if~\(A\) belongs to the bootstrap class; its \(E^2\)\nb-term involves \(\Ext^p_\Cat\bigl(\FK(A), \FK(B)\bigr)\).

The main issue is whether the \(\Ext\)-groups \(\Ext^p_\Cat\bigl(\FK(A), \FK(B)\bigr)\) with \(p\ge2\) vanish, so that our spectral sequence degenerates to an exact sequence of the desired form.  This amounts to checking whether \(\FK(A)\) has a projective resolution of length~\(1\) in~\(\Cat\).

Already for the non-Hausdorff two-point space considered in \cites{Rordam:Classification_extensions, Bonkat:Thesis}, the category~\(\Cat\) has infinite cohomological dimension, that is, there are objects that admit no projective resolution of finite length.  But these objects do not belong to the range of the functor~\(\FK\).  If an \(\Nattrafo\)\nb-module~\(A\) belongs to the range of~\(\FK\), then there are exact sequences
\begin{equation}
  \label{eq:exactness_intro}
  \dotsb \to A(U) \to A(Y) \to A(Y\setminus U) \to A(U) \to
  \dotsb
\end{equation}
for any \(Y\in\Loclo(X)\), \(U\in\Loclo(Y)\) because of~\eqref{eq:six-term_intro}.  But there are \(\Nattrafo\)\nb-modules without finite length projective resolutions.  For totally ordered spaces, an object of~\(\Cat\) has a projective resolution of length~\(1\) if and only if it has a projective resolution of finite length, if and only if the sequences~\eqref{eq:exactness_intro} are exact, if and only if it is the filtrated \(\K\)\nb-theory of some separable \(\Cst\)\nb-algebra over~\(X\), which we can take in the bootstrap class.

For the four-point counterexample considered in Section~\ref{sec:counterexample}, we first find a torsion-free exact module that is not projective, and then use it to find an exact module without projective resolutions of length~\(1\).  Then we find two non-isomorphic objects of the bootstrap class with the same filtrated \(\K\)\nb-theory.  The idea here is to consider a certain exact triangle \(\Sigma C \to A\to B\to C\), which splits on the level of filtrated \(\K\)\nb-theory, so that \(A\oplus C\) and~\(B\) have the same filtrated \(\K\)\nb-theory.  But we can prove in our concrete example that \(A\oplus C\) and~\(B\) are not \(\KK(X)\)\nb-equivalent.

A \(\Cst\)\nb-algebra over the four-point space~\(X\) is a \(\Cst\)\nb-algebra~\(A\) with a distinguished ideal~\(I\) and a direct sum decomposition of~\(A/I\) as a direct sum of three orthogonal ideals.  Since both direct sums and extensions of \(\Cst\)\nb-algebras can be classified by filtrated \(\K\)\nb-theory, it is remarkable that the combination of both provides a counterexample.  Incidentally, the space~\(X^\op\) that corresponds to a \(\Cst\)\nb-algebra~\(A\) with a distinguished ideal~\(I\) and a direct sum decomposition of~\(I\) as a direct sum of three orthogonal ideals also leads to a counterexample in a similar fashion.

For the four-point space~\(X\) above, there is essentially just one module that ought to be projective but is not.  We can add another invariant to filtrated \(\K\)\nb-theory that corresponds to this offending module.  Since this changes our whole category, it may lead to further offending modules, which would have to be added in a second step, and this could, in principle, go on forever.  But in the concrete case at hand, we get projective resolutions of length~\(1\) for all modules over the enriched filtrated \(\K\)\nb-theory.  As a result, the enriched filtrated \(\K\)\nb-theory classifies objects of the bootstrap class over~\(X\) up to \(\KK(X)\)-equivalence, and it classifies purely infinite separable nuclear stable \(\Cst\)\nb-algebras with primitive ideal space~\(X\) and simple subquotients in the bootstrap class.

\subsection{Some basic notation}
\label{sec:notation}

We shall use the following notation from~\cites{Meyer-Nest:Bootstrap}:
\begin{enumerate}[leftmargin=*, widest=\(\Cstarsep(X)\)]
\item[\(\inOb\)] we write \(x\inOb\Cat\) for objects of a category~\(\Cat\) as opposed to morphisms;

\item[\(X\)] topological space, often assumed \emph{sober} (see~\cite{Vickers:Topology_logic});

\item[\(\Open(X)\)] set of open subsets of~\(X\), partially ordered by~\(\subseteq\);

\item[\(\Loclo(X)\)] set of locally closed subsets of~\(X\);

\item[\(\Loclo(X)^*\)] set of connected, non-empty locally closed subsets of~\(X\);

\item[\(\preceq\)] specialisation preorder on~\(X\), defined by \(x\preceq y\) \(\iff\) \(\cl{\{x\}} \subseteq \cl{\{y\}}\)

\item[\(A\)] \(\Cst\)\nb-algebra;

\item[\(\Prim(A)\)] primitive ideal space of~\(A\) with hull--kernel topology;

\item[\(\Ideals(A)\)] set of closed \Star{}ideals in~\(A\), partially ordered by~\(\subseteq\);

\item[\(\Cstarcat(X)\)] category of \(\Cst\)\nb-algebras over~\(X\) with \(X\)\nb-equivariant \Star{}homomorphisms

\item[\(\Cstarsep(X)\)] full subcategory of separable \(\Cst\)\nb-algebras over~\(X\);

\item[\(\KKcat(X)\)] Kasparov category of \(\Cst\)\nb-algebras over~\(X\): its objects are separable \(\Cst\)\nb-algebras over~\(X\), its set of morphisms from~\(A\) to~\(B\) is \(\KK_0(X;A,B)\);

\item[\(\Bootstrap(X)\)] the bootstrap class in \(\KKcat(X)\);

\item[\(i_Y^X\)] extension functor \(\Cstarcat(Y)\to\Cstarcat(X)\) or \(\KKcat(Y)\to\KKcat(X)\) for a subset \(Y\subseteq X\);

\item[\(i_x\)] abbreviation for \(i_{\{x\}}^X\) for \(x\in X\);

\item[\(r_X^Y\)] restriction functor \(\Cstarcat(X)\to\Cstarcat(Y)\) or \(\KKcat(X)\to\KKcat(Y)\) for a locally closed subset \(Y\subseteq X\);

\item[\(\Sigma\)] suspension \(\Sigma A\defeq \CONT_0(\R,A)\).
\end{enumerate}

Roughly speaking, a space is sober if it can be recovered from the lattice \(\Open(X)\).  It is explained in \cite{Meyer-Nest:Bootstrap}*{\S2.5} why we may restrict attention to such spaces.  For finite spaces, sobriety is equivalent to the separation axiom~\(\textup T_0\), that is, two points are equal once they have the same closure.

A \emph{\(\Cst\)\nb-algebra over~\(X\)} is pair \((A,\psi)\) consisting of a \(\Cst\)\nb-algebra~\(A\) and a continuous map \(\psi\colon \Prim(A)\to X\).  If~\(X\) is sober, this is equivalent to a map
\[
\psi^*\colon \Open(X)\to\Ideals(A),
\qquad
U\mapsto A(U),
\]
that preserves finite infima and arbitrary suprema, that is,
\[
A\biggl(\bigcap_{U\in F} U\biggr) = \bigcap_{U\in F} A(U),
\qquad
A\biggl(\bigcup_{U\in S} U\biggr) = \bigvee_{U\in S} A(U)
= \cl{\sum_{U\in S} A(U)},
\]
where \(F\subseteq \Open(X)\) is finite and \(S\subseteq \Open(X)\) is arbitrary.  In particular, this implies \(A(\emptyset)=\{0\}\), \(A(X)=A\), and the monotonicity condition \(A(U)\triangleleft A(V)\) for \(U\subseteq V\).

A \Star{}homomorphism \(f\colon A\to B\) between two \(\Cst\)\nb-algebras over~\(X\) is \emph{\(X\)\nb-equivariant} if \(f\bigl(A(U)\bigr)\subseteq B(U)\) for all \(U\in\Open(X)\).

A subset \(Y\subseteq X\) is \emph{locally closed} if and only if \(Y=U\setminus V\) for open subsets \(V,U\in\Open(X)\) with \(V\subseteq U\).  Then we define \(A(Y)\defeq A(U)/A(V)\) for a \(\Cst\)\nb-algebra~\(A\) over~\(X\); this does not depend on the choice of \(U\) and~\(V\) by \cite{Meyer-Nest:Bootstrap}*{Lemma 2.15}.

If \(Y\subseteq X\) is locally closed and~\(A\) is a \(\Cst\)\nb-algebra over~\(Y\), then we extend~\(A\) to a \(\Cst\)\nb-algebra \(i_Y^XA\) over~\(X\) by \(i_Y^XA(Z) \defeq A(Y\cap Z)\) for \(Z\in\Loclo(X)\).  Conversely, we can restrict a \(\Cst\)\nb-algebra~\(B\) over~\(X\) to a \(\Cst\)\nb-algebra \(r_X^Y(B)\) over~\(Y\) by \(r_X^Y B(Z)\defeq B(Z)\) for all \(Z\in\Loclo(Y)\subseteq\Loclo(X)\).

The category \(\KKcat(X)\) is triangulated, with exact triangles coming either from mapping cone triangles of \(X\)\nb-equivariant \Star{}homomorphisms or, equivalently, from semi-split \(\Cst\)\nb-algebra extensions over~\(X\) (see \cites{Meyer-Nest:BC, Meyer-Nest:Bootstrap}).  Here an extension is called \emph{semi-split} if it splits by an \(X\)\nb-equivariant completely positive contraction.

The \emph{bootstrap class} \(\Bootstrap(X)\) is the localising subcategory of \(\KKcat(X)\) generated by the objects \(i_x\C\) for all \(x\in X\).  That is, it is the smallest class of objects containing these generators that is closed under suspensions, \(\KK(X)\)-equivalence, semi-split extensions, and countable direct sums.

\section{Filtrated \texorpdfstring{$\K$}{K}-theory}
\label{sec:filtrated_K}

Let~\(X\) be a finite topological space.  We do not discuss filtrated \(\K\)\nb-theory for \(\Cst\)\nb-algebras over infinite spaces here.

\begin{definition}
  \label{def:FK_Y}
  For a locally closed subset \(Y\subseteq X\), we define a functor
  \[
  \FK_Y\colon \KKcat(X)\to\Ab^{\Z/2},\qquad
  \FK_Y(A) \defeq \K_*\bigl(A(Y)\bigr).
  \]
  Here \(\Ab\) denotes the category of Abelian groups and \(\Ab^{\Z/2}\) denotes the category of \(\Z/2\)-graded Abelian groups.
\end{definition}

For each \(Y\in\Loclo(X)\), the functor~\(\FK_Y\) is \emph{stable} and \emph{homological}, that is, it intertwines the suspension on \(\KKcat(X)\) with the translation functor on \(\Ab^{\Z/2}\) (this functor shifts the grading), and if \(\Sigma C \to A\to B\to C\) is an exact triangle in \(\KKcat(X)\)---this may, for instance, come from a semi-split extension \(A\into B\prto C\)---then \(\FK_Y(A) \to \FK_Y(B)\to \FK_Y(C)\) is an exact sequence in \(\Ab^{\Z/2}\).

The functors~\(\FK_Y\) together form the filtrated \(\K\)\nb-theory functor.  But the latter also includes its target category, which we now define in a rather abstract way.

\begin{definition}
  \label{def:Nattrafo}
  For \(Y,Z\in\Loclo(X)\), let \(\Nattrafo_*(Y,Z)\) be the \(\Z/2\)-graded Abelian group of all natural transformations \(\FK_Y\Rightarrow \FK_Z\).  The composition of natural transformations provides a product
  \[
  \Nattrafo_i(Y,Z) \times \Nattrafo_j(W,Y)
  \to \Nattrafo_{i+j}(W,Z),
  \qquad f,g\mapsto f\circ g,
  \]
  which is associative and additive in each variable.

  We let~\(\Nattrafo\) be the \(\Z/2\)-graded category whose object set is~\(\Loclo\) and whose morphism space \(Y \to Z\) is \(\Nattrafo_*(Y,Z)\).  The Abelian group structure on these morphism spaces turns this into a pre-additive category.
\end{definition}

\begin{definition}
  \label{def:Nattrafo_module}
  A \emph{module} over~\(\Nattrafo\) is a grading preserving, additive functor \(G\colon \Nattrafo\to\Ab^{\Z/2}\).  That is, it consists of a family of \(\Z/2\)-graded Abelian groups \(G_Y=(G_{Y,0},G_{Y,1})\) for \(Y\in\Loclo(X)\) and product maps
  \[
  \Nattrafo_i(Y,Z) \times G_{Y,j}\to G_{Z,i+j}
  \]
  for all \(Y,Z\in\Loclo(X)\), \(i,j\in\Z/2\); these product maps are associative, additive in each variable, and the identity transformations in \(\Nattrafo(Y,Y)\) act identically on~\(G_Y\) for all \(Y\in\Loclo(X)\).

  Let \(\Mod{\Nattrafo}\) be the category of \(\Nattrafo\)\nb-modules.  The morphisms in \(\Mod{\Nattrafo}\) are the natural transformations of functors or, equivalently, families of grading preserving group homomorphisms \(G_Y\to G'_Y\) that commute with the actions of~\(\Nattrafo\).  Let \(\CMod{\Nattrafo}\) be the full subcategory of countable modules.
\end{definition}

By construction, the natural transformations \(\FK_Y\Rightarrow\FK_Z\) in \(\Nattrafo_*(Y,Z)\) induce maps \(\FK_Y(A)\to\FK_Z(A)\) for all \(A\inOb\KKcat(X)\).  This turns \(\bigl(\FK_Y(A)\bigr)_{Y\in\Loclo(X)}\) into a module over~\(\Nattrafo\).  Furthermore, it is well-known that the \(\K\)\nb-theory of separable \(\Cst\)\nb-algebras such as \(A(Y)\) for \(A\inOb\KKcat(X)\) is countable.

\begin{definition}
  \label{def:filtrated_K}
  \emph{Filtrated \(\K\)\nb-theory} is the functor
  \[
  \FK=(\FK_Y)_{Y\in\Loclo(X)}\colon
  \KKcat(X) \to \CMod{\Nattrafo},
  \qquad
  A \mapsto
  \Bigl(\K_*\bigl(A(Y)\bigr)\Bigr)_{Y\in\Loclo(X)}.
  \]
\end{definition}

The target category \(\CMod{\Nattrafo}\) is an important part of this definition because we will compute groups of morphisms and extensions in this category.

Since \(A(\emptyset)=\{0\}\) for all \(\Cst\)\nb-algebras over~\(X\), we have \(\FK_\emptyset=0\), so that~\(\emptyset\) is a zero object of~\(\Nattrafo\).  Therefore, \(G_\emptyset\) vanishes for any \(\Nattrafo\)\nb-module.

If \(Y\in\Loclo(X)\) is not connected, that is, \(Y=Y_1\sqcup Y_2\) with two disjoint relatively open subsets \(Y_1,Y_2\in\Open(Y)\subseteq \Loclo(X)\), then \(A(Y) \cong A(Y_1)\oplus A(Y_2)\) for any \(\Cst\)\nb-algebra~\(A\) over~\(X\).  Hence \(\FK_Y(A) \cong \FK_{Y_1}(A)\times \FK_{Y_2}(A)\).  The natural transformations that implement this natural isomorphism correspond to a direct sum diagram \(Y\cong Y_1\oplus Y_2\) in~\(\Nattrafo\).  Therefore, any \(\Nattrafo\)\nb-module has \(G_Y\cong G_{Y_1}\oplus G_{Y_2}\); here we use the fact that a functor that is additive on morphisms is also additive on objects, even if the category in question is only pre-additive.

Since~\(X\) is finite, any locally closed subset is a disjoint union of its connected components.  This corresponds to a direct sum decomposition \(Y \cong \bigoplus_{j\in\pi_0(Y)} Y_j\) in~\(\Nattrafo\).  Therefore, we lose no information when we replace \(\Loclo(X)\) by the subset \(\Loclo(X)^*\) of non-empty, connected, locally closed subsets.

\subsection{The representability theorem}
\label{sec:representability}

The representability theorem serves two purposes.  We will first use it to describe the category~\(\Nattrafo\).  Later, we use it to construct geometric resolutions in \(\KKcat(X)\).

\begin{theorem}
  \label{the:representability}
  Let~\(X\) be a finite topological space.  The covariant functors \(\FK_Y\) for \(Y\in\Loclo(X)\) are representable, that is, there are objects \(\Repr_Y\inOb\KKcat(X)\) and natural isomorphisms
  \[
  \KK_*(X;\Repr_Y,A) \cong \FK_Y(A) = \K_*\bigl(A(Y)\bigr)
  \]
  for all \(A\inOb\KKcat(X)\), \(Y\in\Loclo(X)\).
\end{theorem}

Before we prove this theorem in~\S\ref{sec:proof_representability}, we first describe the representing objects~\(\Repr_Y\) explicitly, and we use this to describe the groups of natural transformations \(\Nattrafo_*(Y,Z)\) as \(\K\)\nb-theory groups of certain locally compact spaces.

The construction of~\(\Repr_Y\) requires some preparation.  We equip~\(X\) with the specialisation preorder~\(\preceq\) as in \cite{Meyer-Nest:Bootstrap}*{\S2.7}; recall that \(x\preceq y\) if and only if \(\cl{\{x\}}\subseteq \cl{\{y\}}\).  Since the topological space~\(X\) is finite, it carries the Alexandrov topology of the preorder~\(\preceq\), that is, a subset \(Y\subseteq X\) is open if and only if \(x\succeq y\in Y\) implies \(x\in Y\).  Similarly, \(Y\subseteq X\) is closed if and only if \(x\preceq y\in Y\) implies \(x\in Y\), and locally closed if and only if \(x\preceq y\preceq z\) and \(x,z\in Y\) implies \(y\in Y\).

\begin{definition}
  \label{def:order_complex}
  Let \((X,\preceq)\) be a partially ordered set.  Its \emph{order complex} is the geometric realisation of the simplicial set \(\Ch(X)\) whose \(n\)\nb-simplices are the \emph{chains} \(x_0\preceq x_1\preceq\dotsb\preceq x_n\) in~\(X\) and whose face and degeneracy maps delete or double an entry of the chain.
\end{definition}

Equivalently, \(\Ch(X)\) is the classifying space of the thin category that has object set~\(X\) and a morphism \(x\to y\) whenever \(x\preceq y\).

The order complex is the main ingredient in the construction of the representing objects~\(\Repr_Y\) for \(Y\in\Loclo(X)\).

The \emph{non-degenerate} \(n\)\nb-simplices in \(\Ch(X)\) are the \emph{strict} chains \(x_0\prec\dotsb\prec x_n\) in~\(X\).  We let~\(S_X\) be the set of all \emph{strict} chains.  For each \(I=(x_0\prec\dotsb\prec x_n)\in S_X\), we let~\(\Delta_I\) be a copy of~\(\Delta_n\); more formally, \(\Delta_I=\{(t,I)\mid t\in\Delta_n\}\).  We also let \(\Delta_I^\circ\subseteq\Delta_I\) be the corresponding open simplex \(\Delta_n\setminus\partial\Delta_n\).

The space \(\Ch(X)\) is obtained from the union \(\coprod_{I\in S_X} \Delta_I\) by identifying~\(\Delta_I\) with the corresponding face in~\(\Delta_J\) whenever \(I,J\in S_X\) satisfy \(I\subseteq J\).  Thus the underlying set of \(\Ch(X)\) is a \emph{disjoint} union
\begin{equation}
  \label{eq:nerve_disjoint_union}
  \Ch(X) = \coprod_{I\in S_X} \Delta_I^\circ.
\end{equation}
For \(I\in S_X\), let \(\min I\) and \(\max I\) be the (unique) minimal and maximal elements in~\(S_X\), respectively.  We define two functions
\[
m,M\colon \Ch(X)\to X
\]
by mapping points in \(\Delta_I^\circ\) to \(\min I\) and \(\max I\), respectively.  This well-defines functions on \(\Ch(X)\) because of~\eqref{eq:nerve_disjoint_union}.

\begin{lemma}
  \label{lem:max_min_on_nerve_continuous}
  If \(Y\subseteq X\) is closed, then \(m^{-1}(Y)\) is open and \(M^{-1}(Y)\) is closed in \(\Ch(X)\).  If \(Y\subseteq X\) is open, then \(m^{-1}(Y)\) is closed and \(M^{-1}(Y)\) is open.  If \(Y\subseteq X\) is locally closed, then \(m^{-1}(Y)\) and \(M^{-1}(Y)\) are locally closed.
\end{lemma}

\begin{proof}
  First we show that \(M^{-1}(Y)\) is closed if~\(Y\) is closed.  If \(I\in S_X\) satisfies \(\max I\in Y\), then \(\max J\in Y\) for all \(J\subseteq I\) because \(\max J \preceq \max I\in Y\).  Hence \(\Delta_I\subseteq M^{-1}(Y)\) once \(M^{-1}(Y)\cap \Delta_I^\circ\neq\emptyset\), so that \(M^{-1}(Y)\cap \Delta_I\) is closed for all \(I\in S_X\); this implies that \(M^{-1}(Y)\) is closed.

  A similar argument shows that \(m^{-1}(Y)\) is closed in \(\Ch(X)\) if~\(Y\) is open.  Now the remaining assertions follow easily because the maps \(m^{-1}\) and~\(M^{-1}\) commute with complements, unions, and intersections.
\end{proof}

More explicitly, if \(Y\subseteq X\) is open, then \(m^{-1}(Y)\) is the union of the simplices~\(S_X\) for all chains \(x_0\prec x_1\prec\dotsb\prec x_n\) with \(x_0\in Y\) and hence \(x_0,\dotsc,x_n\in Y\).  Thus
\begin{alignat*}{2}
  m^{-1}(Y) &= \Ch(Y)
  &\qquad&
  \text{if \(Y\subseteq X\) is open.}\\\shortintertext{Similarly,}
  M^{-1}(Y) &= \Ch(Y)
  &\qquad&
  \text{if \(Y\subseteq X\) is closed.}
\end{alignat*}
Here we identify \(\Ch(Y)\) with a subcomplex of \(\Ch(X)\) in the obvious way.

Let~\(X^\op\) be~\(X\) with the topology for the reversed partial order~\(\succ\); that is, the open subsets of~\(X^\op\) are the closed subsets of~\(X\), and vice versa.  We may rephrase Lemma~\ref{lem:max_min_on_nerve_continuous} as follows:

\begin{proposition}
  \label{pro:min_max_continuous}
  The map \((m,M)\colon \Ch(X)\to X^\op\times X\) is continuous.
\end{proposition}

Let
\[
\Repr\defeq \CONT\bigl(\Ch(X)\bigr)
\]
be the \(\Cst\)\nb-algebra of continuous functions on \(\Ch(X)\).  Since
\[
\Prim \Repr = \Prim \CONT\bigl(\Ch(X)\bigr) \cong \Ch(X),
\]
the map \((m,M)\) turns~\(\Repr\) into a \(\Cst\)\nb-algebra over \(X^\op\times X\).  We abbreviate
\[
S(Y,Z) \defeq m^{-1}(Y)\cap M^{-1}(Z)
\subseteq \Ch(X);
\]
this is a locally closed subset of \(\Ch(X)\) by Lemma~\ref{lem:max_min_on_nerve_continuous}

\begin{definition}
  \label{def:Repr_Y}
  We let~\(\Repr_Y\) be the \(\Cst\)\nb-algebra over~\(X\) with
  \[
  \Repr_Y(Z) \defeq \Repr(Y^\op\times Z) = \CONT_0\bigl(S(Y,Z)\bigr)
  \]
  for all \(Y,Z\in\Loclo(X)\); here \(Y^\op\) denotes~\(Y\) with the subspace topology from~\(X^\op\).  Equivalently, we let~\(\Repr_Y\) be the restriction of~\(\Repr\) to \(Y^\op\times X\), viewed as a \(\Cst\)\nb-algebra over~\(X\) via the coordinate projection \(Y^\op\times X\to X\).
\end{definition}

We will prove the Theorem~\ref{the:representability} for this choice of~\(\Repr_Y\) in~\S\ref{sec:proof_representability}.  Taking this for granted, we use the concrete description of~\(\Repr_Y\) to compute the groups of natural transformations.  By the Yoneda Lemma, natural transformations between the functors~\(\FK_Y\) come from morphisms between the representing objects.  More precisely,
\begin{multline}
  \label{eq:nattrafo}
  \Nattrafo_*(Y,Z)
  \cong \KK_*(X;\Repr_Z,\Repr_Y)
  \cong \FK_Z(\Repr_Y)
  = \K_*\bigl(\Repr_Y(Z)\bigr)
  \\ = \K_*\bigl(\Repr(Y^\op \times Z)\bigr)
  = \K^*\bigl(m^{-1}(Y)\cap M^{-1}(Z)\bigr)
  = \K^*\bigl(S(Y,Z)\bigr).
\end{multline}

By the way, the universal property of Kasparov theory says that it makes no difference for the natural transformations \(\FK_Y\Rightarrow\FK_Z\) whether we view these two functors as defined on \(\Cstarsep(X)\) or \(\KKcat(X)\).  But since~\(\Repr_Y\) only represents~\(\FK_Y\) on the level of \(\KKcat(X)\), we get \(\KK_*(X;\Repr_Z,\Repr_Y)\) and not the space of \(X\)\nb-equivariant \Star{}homomorphisms \(\Repr_Z\to\Repr_Y\).

We describe \(S(Y,Z)\) more explicitly using the closure and boundary operations
\begin{alignat*}{2}
  \cl{Z} &\defeq
  \{x\in X\mid \text{there is \(z\in Z\) with
    \(x\preceq z\)}\},
  &\qquad \cl{\partial} Z \defeq \cl{Z}\setminus Z,\\
  \widetilde{Y} &\defeq
  \{x\in X\mid \text{there is \(y\in Y\) with
    \(x\succeq y\)}\},
  &\qquad \widetilde{\partial} Y \defeq \widetilde{Y}\setminus Y.
\end{alignat*}
Of course, \(\cl{Z}\) is the closure of~\(Z\) in~\(X\) and~\(\widetilde{Y}\) is the closure of~\(Y\) in~\(X^\op\).

\begin{lemma}
  \label{lem:describe_Repr_Y_Z}
  If \(Y,Z\in\Loclo(X)\), then
  \[
  S(Y,Z) = \Ch(\widetilde{Y}\cap\cl{Z}) \bigm\backslash
  \bigl(\Ch(\widetilde{Y}\cap\cl{\partial}Z) \cup
  \Ch(\widetilde{\partial}Y\cap \cl{Z})\bigr).
  \]
  In particular,
  \begin{alignat*}{2}
    S(Y,Z) &= \Ch(Y\cap\cl{Z}) \setminus \Ch(Y\cap\cl{\partial}Z)
    &\qquad &\text{if~\(Y\) is open,}\\
    S(Y,Z) &= \Ch(\widetilde{Y}\cap Z) \setminus
    \Ch(\widetilde{\partial}Y\cap Z)
    &\qquad &\text{if~\(Z\) is closed,}\\
    S(Y,Z) &= \Ch(Y\cap Z)
    &\qquad &\text{if~\(Y\) is open and~\(Z\) is closed.}
  \end{alignat*}
\end{lemma}

\begin{proof}
  Let \(x_0\prec x_1\prec\dotsb\prec x_n\) be a strict chain in~\(X\).  The interior of the corresponding simplex belongs to \(S(Y,Z)\) if and only if \(x_0\in Y\) and \(x_n\in Z\).  This implies \(x_j\in \widetilde{Y}\) and \(x_j\in \cl{Z}\) for all~\(j\), so that the simplex belongs to \(\Ch(\widetilde{Y}\cap\cl{Z})\).  Furthermore, we neither have \(x_j\in\widetilde{\partial}Y\cap\cl{Z}\) for all~\(j\) nor \(x_j\in\widetilde{Y}\cap \cl{\partial}Z\) for all~\(j\) because \(x_0\in Y\) and \(x_n\in Z\).  Thus the simplex belongs neither to \(\Ch(\widetilde{Y}\cap\cl{\partial}Z)\) nor to \(\Ch(\widetilde{\partial}Y\cap \cl{Z})\).  Conversely, if \(x_j\in \widetilde{Y}\cap\cl{Z}\) for all~\(j\) and neither \(x_j\in\widetilde{\partial}Y\cap\cl{Z}\) for all~\(j\) nor \(x_j\in\widetilde{Y}\cap \cl{\partial}Z\) for all~\(j\), then some~\(x_j\) must belong to \(Y\cap\cl{Z}\) and some~\(x_k\) must belong to \(\widetilde{Y}\cap Z\).  Since \(Y\cap\cl{Z}\) is closed in \(\widetilde{Y}\cap\cl{Z}\) and \(\widetilde{Y}\cap Z\) is open in \(\widetilde{Y}\cap\cl{Z}\), this implies \(x_0\in Y\) and \(x_n\in Z\).  This shows that the interior of a simplex belongs to \(S(Y,Z)\) if and only if it is contained in \(\Ch(\widetilde{Y}\cap\cl{Z}) \bigm\backslash \bigl(\Ch(\widetilde{Y}\cap\cl{\partial}Z) \cup \Ch(\widetilde{\partial}Y\cap \cl{Z})\bigr)\).
\end{proof}

Lemma~\ref{lem:describe_Repr_Y_Z} and~\eqref{eq:nattrafo} yield
\[
\Nattrafo_*(Y,Z)
\cong \K^*\bigl(S(Y,Z)\bigr)
\cong \K^*\bigl(\Ch(\widetilde{Y}\cap\cl{Z}),
\Ch(\widetilde{Y}\cap\cl{\partial}Z) \cup
\Ch(\widetilde{\partial}Y\cap \cl{Z})\bigr).
\]
This is the \(\K\)\nb-theory of a finite CW-pair and hence is always finitely generated as an Abelian group.

If~\(C\) is any finite simplicial complex, then its barycentric subdivision is of the form \(\Ch(X)\), where~\(X\) is the partially ordered set of non-degenerate simplices in~\(C\).  Thus \(\Nattrafo_*(X,X) = \K^*(\lvert C\rvert)\), so that any finitely generated Abelian group arises as \(\Nattrafo_*(X,X)\).  As a consequence, special properties of the pre-additive category \(\Nattrafo\) can only be hidden in its \emph{composition}.

When we identify \(\Nattrafo_*(Y,Z) \cong \KK_*(X;\Repr_Z,\Repr_Y)\), then the composition of natural transformations corresponds to the Kasparov composition product.  This gets somewhat obscured when we follow the isomorphisms
\[
\KK_*(X;\Repr_Z,\Repr_Y)
\cong \K_*\bigl(\Repr_Y(Z)\bigr)
= \K^*\bigl(S(Y,Z)\bigr).
\]
To describe the composition of natural transformations in terms of \(\K^*\bigl(S(Y,Z)\bigr)\), we must first lift elements of \(\K^*\bigl(S(Y,Z)\bigr)\) back to \(\KK_*(X;\Repr_Z,\Repr_Y)\) and then compose them.  The lifting requires a formula for the natural isomorphism
\begin{equation}
  \label{eq:trafo_Repr_FK}
  \KK_*(X;\Repr_Y,A) \to \K_*\bigl(A(Y)\bigr)
\end{equation}
that occurs in the Representability Theorem.  By the Yoneda Lemma, any such natural transformation is of the form \(f\mapsto f_*(\xi_Y)\) for a unique
\[
\xi_Y\in\K_0\bigl(\Repr_Y(Y)\bigr) = \K^0\bigl(S(Y,Y)\bigr)
= \K^0\bigl(\Ch(Y)\bigr).
\]
The natural transformation in~\eqref{eq:trafo_Repr_FK} is generated by the class of the \(1\)\nb-dimensional trivial vector bundle over the compact space \(\Ch(Y)\) or, equivalently, the class of the unit element in \(\K_0\bigl(\Repr_Y(Y)\bigr)\).

In the examples we consider later, all natural transformations turn out to be products of obvious ones, coming from the \(\K\)\nb-theory six-term exact sequences~\eqref{eq:six-term_intro}.  To check this, we only have to verify that a given element~\(\alpha\) of \(\KK_*(X;\Repr_Z,\Repr_Y)\) lifts a given element of \(\K^*\bigl(S(Y,Z)\bigr)\).  The isomorphism~\eqref{eq:trafo_Repr_FK} maps~\(\alpha\) to \([\xi_Z]\otimes_{\Repr_Z(Z)} \alpha(Z)\) in \(\K_*\bigl(\Repr_Y(Z)\bigr) = \K^*\bigl(S(Y,Z)\bigr)\), where \(\alpha(Z)\) in \(\KK_*\bigl(\Repr_Z(Z),\Repr_Y(Z)\bigr)\) is obtained from~\(\alpha\) by restriction to~\(Z\).  This product is easy to compute.

To get acquainted with this approach to natural transformations, we compute some important examples.  Let \(Y\in\Loclo(X)\) and \(U\in\Open(Y)\).  Since~\(\Repr\) is a \(\Cst\)\nb-algebra over \(X^\op\times X\), there is an extension
\begin{equation}
  \label{eq:Repr_extension}
  \Repr_{Y\setminus U}\into \Repr_Y\prto \Repr_U
\end{equation}
of \(\Cst\)\nb-algebras over~\(X\).  It contains \(\Cst\)\nb-algebra extensions
\[
\Repr_{Y\setminus U}(Z)\into \Repr_Y(Z)\prto \Repr_U(Z)
\]
for all \(Z\in\Loclo(X)\).  Let \(Z\defeq Y\setminus U\).  The extension~\eqref{eq:Repr_extension} is semi-split in \(\Cstarcat(X)\) and hence has a class in \(\KK_1(X;\Repr_U,\Repr_Z)\) and produces an exact triangle
\begin{equation}
  \label{eq:Repr_extension_triangle}
  \Sigma \Repr_U \to \Repr_Z \to \Repr_Y\to \Repr_U
\end{equation}
in \(\KKcat(X)\).

\begin{lemma}
  \label{lem:translate_trafo}
  The maps in the extension triangle~\eqref{eq:Repr_extension_triangle} correspond to the natural transformations \(\FK_U[1]\Leftarrow\FK_Z\Leftarrow\FK_Y\Leftarrow\FK_U\) in~\eqref{eq:six-term_intro}.
\end{lemma}

\begin{proof}
  The natural transformation \(\mu_U^Y\colon \FK_U\Rightarrow\FK_Y\) in~\eqref{eq:six-term_intro} is induced by the natural \Star{}homomorphism \(j\colon A(U)\to A(Y)\).  For \(A=\Repr_U\), this map is invertible because \(S(U,Y)=S(U,U)=\Ch(U)\).  Hence \(j(\xi_U)\in \K^0\bigl(S(U,Y)\bigr)\) is again the class of the trivial vector bundle on~\(\Ch(U)\); this class corresponds to the natural transformation~\(\mu_U^Y\).  The restriction map \(\Repr_Y\prto\Repr_U\) in~\eqref{eq:Repr_extension} maps \([\xi_Y]\) to \([\xi_U]\)---recall that both \([\xi_Y]\) and \([\xi_U]\) are trivial vector bundles.  Hence the restriction map \(\Repr_Y\prto\Repr_U\) and the natural transformation~\(\mu_Y^Z\) correspond to the same class---the \(1\)\nb-dimensional trivial vector bundle on \(\Ch(U)\)---in \(\K^0\bigl(S(U,Y)\bigr)\).

  Similarly, the natural transformation \(\mu_Y^Z\colon \FK_Y\Rightarrow\FK_Z\) is induced by the natural \Star{}homomorphism \(p\colon A(Y)\prto A(Z)\).  For \(A=\Repr_Y\), this is the restriction \Star{}homomorphism \(\CONT\bigl(\Ch(Y)\bigr) \to \CONT\bigl(\Ch(Z)\bigr)\) because \(S(Y,Y)=\Ch(Y)\) and \(S(Y,Z)=\Ch(Z)\).  Since the restriction of a trivial bundle remains trivial, \(\mu_Y^Z\) corresponds to the trivial \(1\)\nb-dimensional vector bundle on \(S(Y,Z)=\Ch(Z)\).  The embedding \(\Repr_Z\prto\Repr_Y\) restricts to an identity map on~\(Z\) because \(S(Z,Z)=S(Z,Y)=\Ch(Z)\).  Since this maps \([\xi_Z]\) to the trivial bundle, the embedding \(\Repr_Z\prto\Repr_Y\) and~\(\mu_Y^Z\) both correspond to the same class---the \(1\)\nb-dimensional trivial vector bundle on \(\Ch(Z)\)---in \(\K^0\bigl(S(Y,Z)\bigr)\).

  Finally, we study the boundary map \(\delta_Z^U\colon \FK_Z\Rightarrow\FK_U[1]\).  We claim that it corresponds to the class of the extension \(\Repr_Z\into\Repr_Y\prto\Repr_U\) in \(\KK_1(X;\Repr_U,\Repr_Z)\).  To prove this, we use that \(\Ch(Y)\) is the join of the spaces \(\Ch(U)\) and \(\Ch(Z)\), so that there is a continuous map \(f\colon \Ch(Y)\to[0,1]\) whose fibres over \(0\) and~\(1\) are \(\Ch(U)\) and \(\Ch(Z)\), respectively.

  More precisely, let \(x_0\prec x_1\prec\dotsb\prec x_n\) be a strict chain in~\(Y\) and let~\(\xi\) be a point of the corresponding simplex with coordinates \((t_0,\dotsc,t_n)\) with \(t_0+\dotsb+t_n=1\), that is, \(\xi= t_0x_0+\dotsb+t_nx_n\).  Then there is \(j\in\{0,\dotsc,n\}\) with \(x_0,\dotsc,x_j\in U\), \(x_{j+1},\dotsc,x_n\in Z\).  We can, therefore, write \(\xi=t_U\xi_U+t_Z\xi_Z\) with
  \begin{alignat*}{2}
    \xi_U &= \frac{t_0x_0+\dotsb+t_jx_j}{t_U} \in
    \Ch(U),&\qquad
    t_U &= t_0+\dotsb+t_j,\\
    \xi_Z &=
    \frac{t_{j+1}x_{j+1}+\dotsb+t_nx_n}{t_Z} \in
    \Ch(Z),&\qquad
    t_Z &= t_{j+1}+\dotsb+t_n.
  \end{alignat*}
  We define a continuous map \(f\colon \Ch(Y)\to[0,1]\) by \(\xi\mapsto t_Z\).  We have
  \[
  S(U,U)=\Ch(U)=f^{-1}(0),\qquad
  S(Z,Z)=\Ch(Z)=f^{-1}(1)
  \]
  by construction, and hence
  \[
  S(Z,U)=\Ch(Y)\setminus\bigl(\Ch(U)\sqcup\Ch(Z)\bigr) =
  f^{-1}\bigl((0,1)\bigr).
  \]

  Now we can compute some boundary maps.  The boundary map
  \[
  \K^0\bigl(S(Z,Z)\bigr) \cong \K_0\bigl(\Repr_Z(Z)\bigr)
  \to \K_1\bigl(\Repr_Z(U)\bigr) \cong \K^1\bigl(S(Z,U)\bigr)
  \]
  maps the class of the trivial bundle \([\xi_Z]\) to \(f^*(\delta)\), where~\(\delta\) denotes a generator of \(\Z\cong\K^1\bigl((0,1)\bigr)\); this follows from the naturality of the boundary map.  The boundary map
  \[
  \K^0\bigl(S(U,U)\bigr) \cong \K_0\bigl(\Repr_U(U)\bigr)
  \to \K_1\bigl(\Repr_Z(U)\bigr) \cong \K^1\bigl(S(Z,U)\bigr)
  \]
  for the extension \(\Repr_Z\into\Repr_Y\prto\Repr_U\) maps the class of the trivial bundle \([\xi_U]\) to \(-f^*(\delta)\), again by naturality of the boundary map.
\end{proof}

\begin{remark}
  \label{rem:fundamental_classes_Repr}
  The proof also describes the classes in \(\K^0\bigl(S(U,Y)\bigr)\), \(\K^0\bigl(S(Y,Z)\bigr)\), and \(\K^1\bigl(S(Z,U)\bigr)\) that correspond to the natural transformations in~\eqref{eq:six-term_intro}.  The natural transformations \(\FK_U\Rightarrow\FK_Y\) and \(\FK_Y\Rightarrow\FK_Z\) are represented by the classes of the trivial vector bundles over the compact spaces \(S(U,Y)\) and \(S(Y,Z)\); the natural boundary map \(\FK_Z\Rightarrow \FK_U[1]\) is represented by \(f^*(\delta)\) for a generator of \(\K^1\bigl((0,1)\bigr)\).
\end{remark}

\subsection{Proof of Theorem~\ref{the:representability}}
\label{sec:proof_representability}

We check first that the natural transformation \(\KK_*(X;\Repr_Y,A) \to \K_*\bigl(A(Y)\bigr)\) induced by~\(\xi_Y\) is an isomorphism if~\(Y\) is the minimal open subset~\(U_x\) containing some point \(x\in X\).  The adjointness relation
\[
\KK_*(X;i_x(A),B) \cong \KK_*\bigl(A,B(U_x)\bigr)
\]
for all \(B\inOb\KKcat(X)\) established in \cite{Meyer-Nest:Bootstrap}*{Proposition 3.12} yields
\[
\KK_*(X;i_x(\C),B) \cong \KK_*\bigl(\C,B(U_x)\bigr)
= \FK_{U_x}(B),
\]
that is, \(i_x(\C)\) represents \(\FK_{U_x}\).  To check that~\(\Repr_{U_x}\) does so as well, we must show that \(i_x(\C)\) and \(\Repr_{U_x}\) are \(\KK(X)\)-equivalent.

Recall that \(i_x(\C)=(\C,x)\), where~\(x\) denotes the map \(\Prim(\C)\cong \{x\}\xrightarrow{\subseteq} X\), and
\[
i_x(\C)(Z) =
\begin{cases}
  \C & \text{if \(x\in Z\),}\\
  0 & \text{otherwise}
\end{cases}
\]
for all \(Z\in\Loclo(X)\).

Since \(U_x = \{y\in X\mid x\preceq y\}\), the preordered set~\(U_x\) has a minimal point, namely~\(x\).  Therefore, the space \(\Ch(U_x)\) is starlike and hence contractible in a canonical way towards~\(x\).  The path from a point in~\(\Delta_I\) for \(I\in S_{U_x}\) to the base point in~\(\Delta_x\) lies in \(\Delta_{I\cup\{x\}}\).  Since \(\max I\cup\{x\} = \max I\), the contraction preserves the ideals \(\Repr_{U_x}(V)\) for \(V\in\Open(X)\), so that we get a homotopy equivalence between \(\CONT\bigl(\Ch(U_x)\bigr)\) and \(i_x(\C)\) in \(\Cstarcat(X)\).  Thus \(\Repr_{U_x}\) corepresents \(\FK_{U_x}\) as well.  It is easy to see that the natural isomorphism \(\KK_*(X;\Repr_{U_x},\blank) \cong \FK_{U_x}\) is induced by~\(\xi_{U_x}\).

Let \(\Good\subseteq\Loclo(X)\) be the set of all \(Z\in\Loclo(X)\) for which the natural transformation \(\KK_*(X;\Repr_Z,A)\to \FK_Z(A)\) induced by~\(\xi_Z\) is an isomorphism.  We must show \(\Good=\Loclo(X)\).  We have just seen that \(U_x\in \Good\) for all \(x\in X\).

Let \(Y\in\Loclo(X)\) and \(U\in\Open(Y)\); we claim that all three of \(U\), \(Y\), and \(Y\setminus U\) are good once two of them are.  This follows from the Five Lemma because the maps induced by~\(\xi_Z\) for \(Z=U,Y,Y\setminus U\) intertwine the maps in the six-term exact sequences~\eqref{eq:six-term_intro} and
\[
\xymatrix{
  \KK_0(X;\Repr_U,A) \ar[r]&
  \KK_0(X;\Repr_Y,A) \ar[r]&
  \KK_0(X;\Repr_{Y\setminus U},A) \ar[d]\\
  \KK_1(X;\Repr_{Y\setminus U},A) \ar[u]&
  \KK_1(X;\Repr_Y,A) \ar[l]&
  \KK_1(X;\Repr_U,A) \ar[l]
}
\]
for any \(A\inOb\KKcat(X)\); the latter six-term exact sequence is induced by the semi-split extension~\eqref{eq:Repr_extension_triangle}.  The commutativity of the relevant diagrams follows from the computations in the proof of Lemma~\ref{lem:translate_trafo} (which do not depend on Theorem~\ref{the:representability}).

The two-out-of-three property of~\(\Good\) implies:
\[
U,V\in\Open(X),\quad U,V,U\cap V\in\Good
\qquad\Longrightarrow\qquad
U\cup V\in \Good
\]
because \((U\cup V)\setminus U = V\setminus (U\cap V)\).  By induction on the length of~\(U\), this implies that all open subsets of~\(X\) belong to~\(\Good\).  Since any locally closed subset is a difference of two open subsets, we conclude that \(\Good=\Loclo(X)\).  This finishes the proof of Theorem~\ref{the:representability}.

\section{An example}
\label{sec:example_Nattrafo}

In this section, we restrict our attention to a special class of spaces, namely, the spaces \(X=\{1,\dotsc,n\}\) totally ordered by~\(\le\) for \(n\in\N\).  We let
\[
[a,b]\defeq \{x\in X\mid a\le x\le b\}.
\]
for \(a,b\in\Z\).  We equip~\(X\) with the Alexandrov topology, so that the open subsets are \([a,n]\) for all \(a\in X\); the closed subsets are \([1,b]\) with \(b\in X\), and the locally closed subsets are those of the form \([a,b]\) with \(a,b\in X\) and \(a\le b\).  Any locally closed subset of~\(X\) is connected.

\subsection{Computations with the order complex}
\label{sec:compute_order_complex}

Since any subset of~\(X\) is totally ordered, the space \(\Ch([a,b])\) is just a closed simplex of dimension \(b-a\) for any \(b\ge a\).  We denote the corresponding face of \(\Ch(X)\) by \(\Delta_{[a,b]}\).  This is understood to be empty for \(a>b\).

From now on, we let
\[
Y=[a_1,b_1],\qquad
Z=[a_2,b_2],\qquad
\text{with \(1\le a_1\le b_1\le n\) and
  \(1\le a_2\le b_2\le n\).}
\]
Then \(\widetilde{Y} = [a_1,n]\), \(\widetilde{\partial}Y = [b_1+1,n]\), \(\cl{Z} = [1,b_2]\), and \(\cl{\partial}Z = [1,a_2-1]\).  Lemma~\ref{lem:describe_Repr_Y_Z} yields
\[
S(Y,Z) = \Delta_{[a_1, b_2]}
\setminus \bigl(\Delta_{[a_1, a_2-1]}
\cup \Delta_{[b_1+1, b_2]}\bigr).
\]
Now we distinguish three cases:
\begin{enumerate}[label=\emph{Case} \arabic{*}:]
\item If \(a_2\le a_1 \le b_2\le b_1\), then \(S(Y,Z) = \Delta_{[a_1,b_2]}\) is a non-empty closed simplex.  Hence \(\Nattrafo_*(Y,Z) \cong \K^*\bigl(S(Y,Z)\bigr) \cong \Z[0]\) (this means~\(\Z\) in degree~\(0\)).

\item If \(a_2-1\le b_1\), \(a_1<a_2\), and \(b_1<b_2\), then \(S(Y,Z)\) is obtained from a closed simplex by removing two disjoint, non-empty closed faces.  Excision yields \(\Nattrafo_*(Y,Z) \cong \K^*\bigl(S(Y,Z)\bigr) \cong \Z[1]\) (this means~\(\Z\) in degree~\(1\)).

\item In all other cases, \(S(Y,Z)\) is either empty, a difference of two closed simplices, or a difference \(\sigma\setminus(\tau_1\cup\tau_2)\) for two non-empty closed faces \(\tau_1\) and~\(\tau_2\) of a simplex~\(\sigma\) that intersect.  Then \(\tau_1\cup\tau_2\) and~\(\sigma\) are both contractible, so that \(\Nattrafo_*(Y,Z) \cong \K^*\bigl(S(Y,Z)\bigr) \cong 0\).
\end{enumerate}

Summing up, we get
\begin{equation}
  \label{eq:Nattrafo_example}
  \Nattrafo_*(Y,Z) =
  \begin{cases}
    \Z[0] &\text{if \(a_2\le a_1 \le b_2\le b_1\),}\\
    \Z[1] &\text{if \(a_2-1\le b_1\), \(a_1<a_2\), and
      \(b_1<b_2\),}\\
    0 &\text{otherwise.}
  \end{cases}
\end{equation}

\subsection{Products of natural transformations}
\label{sec:products_natural_transformations}

Our next task is to identify the natural transformations that correspond to the generators of the groups in~\eqref{eq:Nattrafo_example}; this also allows us to compute products in~\(\Nattrafo\).

First we study the grading preserving transformations that appear in the first case.  We introduce a partial order~\(\ge\) and a strict partial order~\(\gg\) on \(\Loclo(X)\) by
\begin{alignat*}{3}
  [a_1,b_1] &\ge [a_2,b_2]
  &\qquad\iff\qquad&
  \text{\(a_1\ge a_2\) and \(b_1\ge b_2\),}\\
  [a_1,b_1] &\gg [a_2,b_2]
  &\qquad\iff\qquad&
  \text{\(a_1>b_2\).}
\end{alignat*}
Our computation shows that \(\Nattrafo_0(Y,Z)\neq\{0\}\) if and only if \(Y\ge Z\) but not \(Y\gg Z\).  This is equivalent to \(Y\cap Z\) being non-empty, closed in~\(Y\), and open in~\(Z\).  Under these assumptions, there is a natural non-zero \Star{}homomorphism given by the composition
\[
\mu_Y^Z\colon A(Y)\prto A(Y\cap Z) \into A(Z)
\]
because \(A(Y\cap Z)\) is a quotient of \(A(Y)\) and an ideal in \(A(Z)\).  The natural transformation \(\FK_Y\Rightarrow\FK_Z\) induced by~\(\mu_Y^Z\) maps \(\xi_Y\in \FK_{Y,0}(\Repr_Y)\), which is the class of the trivial line bundle over \(S(Y,Y)=\Delta_{[a_1,b_1]}\), to the trivial line bundle over \(S(Y,Z)= \Delta_{[a_1,b_2]}\).  Since this is the generator of \(\FK_{Z,0}(\Repr_Y) = \K^0\bigl(S(Y,Z)\bigr) \cong \Z[0]\), the natural transformation~\(\mu_Y^Z\) generates \(\Nattrafo_0(Y,Z)\).

If \(Y\gg Z\), then we let \(\mu_Y^Z\colon A(Y)\to A(Z)\) be the zero map, which induces the zero transformation \(\FK_Y\Rightarrow\FK_Z\).  With this convention, we get \(\mu_Y^Z\circ \mu_W^Y=\mu_W^Z\) for all \(Y,Z,W\in\Loclo(X)\) with \(W\ge Y\ge Z\), also if \(W\gg Z\); this equation holds on the level of \Star{}homomorphisms and, therefore, also for the induced natural transformations.  We can sum this up as follows:

\begin{lemma}
  \label{lem:Nattrafo_even}
  The category~\(\Nattrafo_0\) of grading-preserving natural transformations \(\FK_Y\Rightarrow\FK_Z\) for \(Y,Z\in\Loclo(X)\) is the pre-additive category generated by natural transformations \(\mu_Y^Z\colon \FK_Y\Rightarrow\FK_Z\) for all \(Y\ge Z\) with the relations \(\mu_Y^Z\circ\mu_W^Y=\mu_W^Z\) for \(W\ge Y\ge Z\) and \(\mu_Y^Z=0\) for \(Y\gg Z\).
\end{lemma}

This list of generators is longer than necessary.  Clearly, we can write any \(\mu_Y^Z\) as a product of the transformations \(\mu_{[a,b]}^{[a-1,b]}\) for \(2\le a\le b\le n\) and \(\mu_{[a,b]}^{[a,b-1]}\) for \(1\le a<b\le n\).  Moreover, these transformations themselves are indecomposable, that is, they cannot be written themselves as products in a non-trivial way.

Now we turn to the natural transformations of degree~\(1\).  For any \(b\in X\) and any \(\Cst\)\nb-algebra~\(A\) over~\(X\), we have a natural \(\Cst\)\nb-algebra extension
\[
A([b,n]) \into A([1,n]) \prto A([1,b-1]),
\]
which generates an odd natural transformation
\[
\delta_b\colon \FK_{[1,b-1]} \Rightarrow \FK_{[b,n]}.
\]
Composing with the grading preserving natural transformations~\(\mu\) above, we get a natural transformation of degree~\(1\)
\begin{equation}
  \label{eq:odd_nattrafo}
  \delta_Y^Z\colon
  \FK_Y = \FK_{[a_1,b_1]} \xRightarrow{\mu}
  \FK_{[1,a_2-1]} \xRightarrow{\delta_{a_2}}
  \FK_{[a_2,n]} \xRightarrow{\mu}
  \FK_{[a_2,b_2]} = \FK_Z
\end{equation}
whenever \(b_1\ge a_2-1\).

Equation~\eqref{eq:Nattrafo_example} predicts that this transformation vanishes if \(a_1\ge a_2\) or \(b_1\ge b_2\).  This can be verified as follows.  Vanishing for \(a_1\ge a_2\) is clear because then \([a_1,b_1] \gg [1,a_2-1]\).  By the naturality of the boundary map, the transformation in~\eqref{eq:odd_nattrafo} agrees with the composition of \(\mu\colon \FK_{[a_1,b_1]} \Rightarrow \FK_{[a_1,a_2-1]}\) with the boundary map for the extension
\begin{equation}
  \label{eq:extension_II}
  A([a_2,b_2]) \into A([a_1,b_2]) \prto A([a_1,a_2-1]).
\end{equation}
If \(b_1\ge b_2\), then \(\mu_{[a_1,b_1]}^{[a_1,a_2-1]}\) factors through the quotient map in~\eqref{eq:extension_II}.  But the composite of two maps in a six-term exact sequence vanishes.

Equation~\eqref{eq:odd_nattrafo} produces a natural transformation \(\delta_Y^Z\in\Nattrafo_1(Y,Z)\) whenever \(a_1<a_2\), \(b_1<b_2\), and \(a_2-1\le b_1\), that is, whenever~\eqref{eq:Nattrafo_example} predicts \(\Nattrafo_1(Y,Z)\) to be non-zero.  We claim that~\(\delta_Y^Z\) generates this group.  This follows because the natural transformation~\(\delta_Y^Z\) maps the class of the trivial line bundle over \(S(Y,Y)\) to the generator of \(\K^1\bigl(S(Y,Z)\bigr) \cong \Z\).

Notice that \(\Nattrafo_1([a_2,n],Z)=\{0\}\) for any \(Z\in\Loclo(X)\).  Since the natural transformation~\eqref{eq:odd_nattrafo} above factors through \(\FK_{[a_2,n]}\), \emph{any product of two odd natural transformations vanishes}.  Thus the category~\(\Nattrafo\) is a split extension of~\(\Nattrafo_0\) by the bimodule~\(\Nattrafo_1\).  The bimodule structure on~\(\Nattrafo_1\) is very simple: a product \(\mu_Y^Z\circ\delta_W^Y\) or \(\delta_Y^Z\circ\mu_W^Y\) is equal to~\(\delta_W^Z\) whenever all three natural transformations are defined, and zero otherwise.

\begin{example}
  \label{exa:Cstar_extensions}
  To make our constructions more concrete, we now consider the example \(n=2\), which corresponds to extensions of \(\Cst\)\nb-algebras.  There are only three non-empty locally closed subsets: \(1=[1,1]\), \(12=[1,2]\), and \(2=[2,2]\).  The order complex is an interval; we label its end points \(1\) and~\(2\).  The map \((m,M)\) from \(\Ch(X)=[1,2]\) to \(X^\op\times X\) maps
  \[
  1\mapsto (1,1),\qquad
  2\mapsto (2,2),\qquad
  \mathopen]1,2\mathclose[ \mapsto (1,2).
  \]
  Correspondingly, we have
  \begin{alignat*}{3}
    S(1,1)&=\{1\},&\qquad
    S(1,2)&=\mathopen]1,2\mathclose[,&\qquad
    S(1,12)&=[1,2\mathclose[,\\
    S(2,1)&=\emptyset,&\qquad
    S(2,2)&=\{2\},&\qquad
    S(2,12)&=\{2\},\\
    S(12,1)&=\{1\},&\qquad
    S(12,2)&=\mathopen]1,2],&\qquad
    S(12,12)&=[1,2].
  \end{alignat*}
  Taking \(\K\)\nb-theory, we get
  \begin{alignat*}{3}
    \Nattrafo(1,1)&=\Z[0],&\qquad
    \Nattrafo(1,2)&=\Z[1],&\qquad
    \Nattrafo(1,12)&=0,\\
    \Nattrafo(2,1)&=0,&\qquad
    \Nattrafo(2,2)&=\Z[0],&\qquad
    \Nattrafo(2,12)&=\Z[0],\\
    \Nattrafo(12,1)&=\Z[0],&\qquad
    \Nattrafo(12,2)&=0,&\qquad
    \Nattrafo(12,12)&=\Z[0].
  \end{alignat*}
\end{example}

\subsection{Ring-theoretic properties of the natural transformations}
\label{sec:ring-theoretic_properties}

We now observe some general ring-theoretic properties of \(\Nattrafo\) for \(X=\{1,\dotsc,n\}\) with the total order.  We exclude the trivial case \(n=1\).  We may replace \(\Nattrafo\) by a \(\Z/2\)-graded ring by taking the direct sum of \(\Nattrafo_*(Y,Z)\) for all \(Y,Z\in\Loclo(X)^*\) and defining the product as usual for a category ring.  Then \(\Nattrafo\)\nb-modules become \(\Z/2\)-graded modules over this \(\Z/2\)-graded ring, and ring-theoretic notions such as the Jacobson radical and the balanced tensor product \(\otimes_\Nattrafo\) make sense.

\begin{definition}
  \label{def:Nattrafo_nil}
  Let \(\Nattrafo_\nil\subseteq\Nattrafo\) be the subgroup spanned by the natural transformations~\(\mu_Y^Z\) with \(Y\neq Z\) and~\(\delta_Y^Z\) with arbitrary \(Y,Z\).

  Let \(\Nattrafo_\sesi\subseteq\Nattrafo\) be the subgroup spanned by the natural transformations~\(\mu_Y^Y\) with \(Y\in\Loclo(X)^*\).
\end{definition}

\begin{lemma}
  \label{lem:Nattrafo_nil}
  The subgroup \(\Nattrafo_\nil\) is the maximal nilpotent ideal in~\(\Nattrafo\), it is the nilradical and the Jacobson radical of \(\Nattrafo\).  The subgroup \(\Nattrafo_\sesi\) is a semi-simple subring, and \(\Nattrafo\) decomposes as a semi-direct product \(\Nattrafo_\nil \rtimes \Nattrafo_\sesi\).
\end{lemma}

\begin{proof}
  Since all~\(\mu_Y^Y\) are idempotent, \(\Nattrafo_\sesi\) is a subring isomorphic to \(\Z^{\Loclo(X)^*}\) with pointwise multiplication.  It is easy to see that \(\Nattrafo_\nil\) is an ideal in \(\Nattrafo\).  It is nilpotent, that is, \(\Nattrafo_\nil^k=\{0\}\) for some \(k\in\N\), because \(\Loclo(X)^*\) is finite and~\(\ge\) is a partial order on it.  Since \(\Nattrafo=\Nattrafo_\nil\oplus\Nattrafo_\sesi\) as Abelian groups, we get the desired semi-direct product decomposition.  Since the Jacobson radical of \(\Nattrafo_\sesi\) vanishes, \(\Nattrafo_\nil\) is both the nilradical and the Jacobson radical of \(\Nattrafo\).
\end{proof}

We are going to use Lemma~\ref{lem:Nattrafo_nil} to characterise the projective \(\Nattrafo\)\nb-modules.  This characterisation involves the following two definitions.

\begin{definition}
  \label{def:module_exact}
  We call an \(\Nattrafo\)\nb-module~\(M\) \emph{exact} if the chain complexes
  \[
  \dotsb \to M(U)\xrightarrow{\mu_U^Y}
  M(Y)\xrightarrow{\mu_Y^{Y\setminus U}}
  M(Y\setminus U)\xrightarrow{\delta_{Y\setminus U}^U} M(U) \to
  \dotsb
  \]
  are exact for all \(Y\in\Loclo(X)\), \(U\in\Open(Y)\) as in~\eqref{eq:exactness_intro}.
\end{definition}

\begin{proposition}
  \label{pro:exact_modules}
  Let \(K\into E\prto Q\) be an extension of \(\Nattrafo\)\nb-modules.  If two of the modules \(K,E,Q\) are exact, so is the third one.
\end{proposition}

\begin{proof}
  Given \(U\) and~\(Y\) as above and a module~\(M\), let \(C_\bullet(M)\) be the chain complex
  \[
  \dotsb \to M(U)[m] \to M(Y)[m] \to M(Y\setminus U)[m] \to
  M(U)[m-1] \to \dotsb.
  \]
  Then \(C_\bullet(K) \into C_\bullet(E) \prto C_\bullet(Q)\) is an extension of chain complexes.  The long exact homology sequence shows that all three of these chain complexes are exact once two of them are exact.
\end{proof}

\begin{definition}
  \label{def:module_semi-simple}
  Given an \(\Nattrafo\)-module~\(M\), we let
  \[
  \Nattrafo_\nil\cdot M = \{x\cdot m\mid x\in\Nattrafo_\nil,\
  m\in M\},\qquad
  M_\sesi \defeq M/\Nattrafo_\nil\cdot M.
  \]
  We call \(M_\sesi\) the \emph{semi-simple part of~\(M\)}.
\end{definition}

Since the tensor product over~\(\Nattrafo\) is right exact, \(M_\sesi \cong \Nattrafo_\sesi \otimes_\Nattrafo M\).  We need the following more concrete description of \(M_\sesi\) or, equivalently, of \(\Nattrafo_\nil\cdot M\).

\begin{lemma}
  \label{lem:gen_M_nil}
  Let~\(M\) be an \(\Nattrafo\)\nb-module and let \(Y=[a,b]\) with \(1\le a \le b\le n\).  Then
  \[
  (\Nattrafo_\nil\cdot M)(Y) =
  \begin{cases}
    \mu_{[a+1,b]}^Y(M[a+1,b]) + \mu_{[a,b+1]}^Y(M[a,b+1])
    &\text{if \(a<b<n\)},\\
    \mu_{[a,b+1]}^Y(M[a,b+1])
    &\text{if \(a=b<n\)},\\
    \mu_{[a+1,b]}^Y(M[a+1,b]) +
    \delta_{[1,a-1]}^Y (M[1,a-1])
    &\text{if \(1<a<b=n\)},\\
    \mu_{[a+1,b]}^Y(M[a+1,b])
    &\text{if \(1=a<b=n\)},\\
    \delta_{[1,a-1]}^Y (M[1,a-1])
    &\text{if \(a=b=n\).}
  \end{cases}
  \]
  If~\(M\) is exact, then
  \[
  (\Nattrafo_\nil\cdot M)(Y) =
  \begin{cases}
    \ker \bigl(\delta_{[a,b]}^{[a+1,b+1]}\colon M[a,b]\to
    M[a+1,b+1]\bigr)
    &\text{if \(b<n\)},\\
    \ker \bigl(\mu_{[a,b]}^{[1,a]}\colon M[a,b]\to
    M[1,a]\bigr)
    &\text{if \(b=n\)}.
  \end{cases}
  \]
\end{lemma}

\begin{proof}
  The first assertion holds because any natural transformation \(\FK_Z\Rightarrow\FK_Y\) with \(Z\neq Y\) factors through \(\mu_{[a+1,b]}^Y\) or \(\mu_{[a,b+1]}^Y\) if \(a<b<n\), through \(\mu_{[a,b+1]}^Y\) if \(a=b<n\), and so on.  Here we use that the natural transformations \(\mu_{[a,b]}^{[a-1,b]}\) for \(2\le a\le b\le n\), \(\mu_{[a,b]}^{[a,b-1]}\) for \(1\le a< b\le n\), and \(\delta_{[1,a-1]}^{[a,n]}\) for \(2\le a\le n\) already generate \(\Nattrafo_*\), that is, all other transformations \(\mu_Y^Z\) or \(\delta_Y^Z\) with \(Y\neq Z\) can be written as products of these generators.  By the way, these natural transformations even form a basis for the subquotient \(\Nattrafo_\nil/\Nattrafo_\nil^2\).

  Now assume that~\(M\) is exact.  If \(a=b<n\), then
  \[
  (\Nattrafo_\nil\cdot M)[a,a]
  = \range \Bigl(\mu_{[a,a+1]}^{[a,a]}\Bigr)
  = \ker \Bigl(\delta_{[a,a]}^{[a+1,a+1]}\Bigr).
  \]
  Similarly, we get
  \[
  (\Nattrafo_\nil\cdot M)[n,n] = \ker
  \Bigl(\mu_{[n,n]}^{[1,n]}\Bigr),
  \qquad
  (\Nattrafo_\nil\cdot M)[1,n] = \ker
  \Bigl(\mu_{[1,n]}^{[1,1]}\Bigr).
  \]

  Given \(f_1\colon A_1\to B\) and \(f_2\colon A_2\to B\) and two exact sequences
  \[
  A_1 \xrightarrow{f_1} B \xrightarrow{g_1} C_1,
  \qquad
  A_2 \xrightarrow{g_1f_2} C_1 \xrightarrow{g_2} C_2,
  \]
  we have
  \begin{multline}
    \label{eq:range_sum}
    \range(f_1)+\range(f_2)
    = \ker(g_1)+\range(f_2)
    \\= \{x\in B\mid
    g_1(x) \in \range(g_1\circ f_2) = \ker(g_2)\}
    = \ker(g_2\circ g_1).
  \end{multline}
  If \(a<b<n\), then we apply this to the maps on~\(M\) induced by \(f_1=\mu_{[a+1,b]}^Y\) and \(f_2=\mu_{[a,b+1]}^Y\) with \(Y=[a,b]\).  We get \(g_1 = \mu_Y^{[a,a]}\), \(g_1\circ f_2= \mu_{[a,b+1]}^{[a,a]}\), and hence \(g_2=\delta_{[a,a]}^{[a+1,b+1]}\) and \(g_2\circ g_1=\delta_{[a,b]}^{[a+1,b+1]}\).  This yields the desired formula for \((\Nattrafo_\nil\cdot M)[a,b]\) for \(a<b<n\), using the exactness of~\(M\).  If \(a<b=n\), then we apply the same reasoning to \(f_1=\mu_{[a+1,b]}^Y\) and \(f_2=\delta_{[1,a-1]}^Y\).  Here we get \(g_1=\mu_Y^{[a,a]}\) as above, \(g_1\circ f_2=\delta_{[1,a-1]}^{[a,a]}\), and hence \(g_2=\mu_{[a,a]}^{[1,a]}\) and \(g_2\circ g_1 = \mu_{[a,b]}^{[1,a]}\).  This yields the desired formula for \((\Nattrafo_\nil\cdot M)[a,b]\) for \(a<b=n\).
\end{proof}

\begin{remark}
  \label{rem:longest}
  The natural transformation \(\delta_{[a,b]}^{[a+1,b+1]}\) for \(b<n\) or \(\mu_{[a,n]}^{[1,a]}\) for \(b=n\) is the longest natural transformation out of \([a,b]\) in the following sense: it factors through \(\delta_{[a,b]}^Z\) or \(\mu_{[a,b]}^Z\) whenever the latter is defined and non-zero.  Thus Lemma~\ref{lem:gen_M_nil} identifies \(\Nattrafo_\nil\cdot M(Y)\) with the largest proper subgroup of \(M(Y)\) that is the kernel of some \(\delta_{[a,b]}^Z\) or \(\mu_{[a,b]}^Z\).
\end{remark}

The following proposition is a rather trivial variant of the \emph{Nakayama Lemma}.  Unlike in the usual Nakayama Lemma, we do not assume the module to be finitely generated.  This is no problem because the relevant ideal \(\Nattrafo_\nil\) is nilpotent.

\begin{proposition}
  \label{pro:Nakayama}
  Let~\(M\) be an \(\Nattrafo\)\nb-module with \(M_\sesi=0\).  Then \(M=0\).
\end{proposition}

\begin{proof}
  By assumption, \(M=\Nattrafo_\nil\cdot M\).  By induction, this implies \(M=\Nattrafo_\nil^j\cdot M\) for all \(j\in\N\).  Since \(\Nattrafo_\nil^k=0\) for some~\(k\), we get \(M=0\).
\end{proof}

\subsection{Characterisation of free and projective modules}
\label{sec:free_projective_modules}

\begin{definition}
  \label{def:free_Nattrafo-module}
  For \(Y\in\Loclo(X)\), the \emph{free \(\Nattrafo\)\nb-module on~\(Y\)} is defined by
  \[
  P_Y(Z)\defeq \Nattrafo_*(Y,Z) \qquad
  \text{for all \(Z\in\Loclo(X)\).}
  \]
  An \(\Nattrafo\)\nb-module is called \emph{free} if it is isomorphic to a direct sum of degree-shifted free modules \(P_Y[j]\), \(j\in \Z/2\).
\end{definition}

\begin{theorem}
  \label{the:projective_module_Nattrafo}
  Let~\(M\) be an \(\Nattrafo\)\nb-module.  Then the following are equivalent:
  \begin{enumerate}[label=\textup{(\roman{*})}]
  \item \(M\) is a free \(\Nattrafo\)\nb-module.

  \item \(M\) is a projective \(\Nattrafo\)\nb-module.

  \item \(M_\sesi(Y)= \Nattrafo_\sesi\otimes_\Nattrafo M(Y)\) is a free Abelian group for all \(Y\in\Loclo(X)\) and \(\Tor_1^\Nattrafo(\Nattrafo_\sesi, M)=0\).

  \item \(M(Y)\) is a free Abelian group for all \(Y\in\Loclo(X)\) and~\(M\) is exact.
  \end{enumerate}
  Here \(\Tor_1^\Nattrafo\) denotes the first derived functor of \(\otimes_\Nattrafo\).  The first three conditions remain equivalent when we replace~\(\Nattrafo\) by any ring that is a nilpotent extension of the ring~\(\Z^N\) for some \(N\in\N\).
\end{theorem}

\begin{proof}
  The Yoneda Lemma asserts that \(\Hom(P_Y,M) \cong M(Y)\) for all \(Y\in\Loclo(X)\) and all \(\Nattrafo\)\nb-modules~\(M\).  Hence free modules are projective, that is, (1)\(\Longrightarrow\)(2).  A functor of the form \(M\mapsto R\otimes_S M\) for a ring homomorphism \(S\to R\) always maps free modules to free modules and hence maps projective modules to projective modules.  Furthermore, derived functors like \(\Tor_1^\Nattrafo\) automatically vanish on projective modules.  This yields the implication (2)\(\Longrightarrow\)(3).  We are going to prove that~(3) implies~(1).

  Since \(M_\sesi(Y)\) is a free Abelian group for all~\(Y\), \(M_\sesi\) is a free module over \(\Nattrafo_\sesi \cong \Z^{\Loclo(X)^*}\).  Hence \(P\defeq \Nattrafo\otimes_{\Nattrafo_\sesi} M_\sesi\) is a free \(\Nattrafo\)\nb-module.  The canonical projection \(M\to M_\sesi\) splits by an \(\Nattrafo_\sesi\)\nb-module homomorphism because \(M_\sesi\) is free.  This induces an \(\Nattrafo\)\nb-module homomorphism \(f\colon P\to M\) because of the adjointness relation
  \[
  \Hom_\Nattrafo(\Nattrafo\otimes_{\Nattrafo_\sesi} X,Y)
  \cong \Hom_{\Nattrafo_\sesi}(X,Y).
  \]
  We claim that~\(f\) is invertible, so that \(M\cong P\) is a free module as asserted.  We have
  \[
  P_\sesi
  = \Nattrafo_\sesi \otimes_\Nattrafo \Nattrafo
  \otimes_{\Nattrafo_\sesi} M_\sesi
  \cong \Nattrafo_\sesi \otimes_{\Nattrafo_\sesi} M_\sesi
  \cong M_\sesi.
  \]
  Inspection shows that this isomorphism is induced by~\(f\).  Since the functor \(M\mapsto M_\sesi\) is right-exact, this implies \(\coker(f)_\sesi=0\) and hence \(\coker(f)=0\) by the Nakayama Lemma (Proposition~\ref{pro:Nakayama}).  That is, \(f\) is an epimorphism.

  Let \(K\defeq \ker(f)\), then we get an exact sequence of \(\Nattrafo\)\nb-modules \(K\into P\prto M\).  The derived functors of \(\Nattrafo_\sesi \otimes_\Nattrafo\blank\) provide a long exact sequence
  \begin{equation}
    \label{eq:Tor_sequence}
    0 \to \Tor^\Nattrafo_1(\Nattrafo_\sesi,M)
    \to K_\sesi \to P_\sesi \xrightarrow[\cong]{f} M_\sesi \to 0.
  \end{equation}
  This exact sequence ends at \(\Tor_1^\Nattrafo(\Nattrafo_\sesi,P)=0\) because~\(P\) is projective.  Since \(\Tor^\Nattrafo_1(\Nattrafo_\sesi,M)=0\) by assumption, we conclude that \(K_\sesi=0\).  Hence another application of the Nakayama Lemma shows that \(\ker(f)=0\) as well.  Thus~\(f\) is invertible.  This finishes the proof of the implication (3)\(\Longrightarrow\)(1), showing that the first three conditions are equivalent.  Furthermore, our argument so far works for any split nilpotent extension of~\(\Z^N\) for some \(N\in\N\) because this is the only information about~\(\Nattrafo\) that we have used.  Nilpotent extensions of the ring~\(\Z^N\) always split because we can lift orthogonal idempotents in nilpotent extensions.

  Free \(\Nattrafo\)\nb-modules are exact, and they consist of free Abelian groups by~\eqref{eq:Nattrafo_example}.  This yields the implication (1)\(\Longrightarrow\)(4).  We are going to prove that~(4) implies~(3).  This will finish the proof of the theorem.  Since we will use this once again later, we state half of this argument as a separate lemma:

  \begin{lemma}
    \label{lem:Tor_technical}
    Let~\(M\) be an exact \(\Nattrafo\)\nb-module.  Then \(\Tor_1^\Nattrafo(\Nattrafo_\sesi, M)=0\).
  \end{lemma}

  \begin{proof}
    Let \(\pi\colon P\to M\) be an epimorphism with a projective \(\Nattrafo\)\nb-module~\(P\), and let \(K\defeq \ker \pi\).  Since projective modules are exact and \(K\into P\prto M\) is a module extension, Proposition~\ref{pro:exact_modules} shows that~\(K\) is exact.  We still have an exact sequence as in~\eqref{eq:Tor_sequence}.

    Since \(K\) and~\(P\) are exact, Lemma~\ref{lem:gen_M_nil} identifies \(K_\sesi(Y)\) and \(P_\sesi(Y)\) in a natural way with subspaces of \(K(Z)\) and \(P(Z)\) for suitable~\(Z\); here we use \(A/\ker(f) \cong \range(f)\) for a group homomorphism \(f\colon A\to B\).  Since the map \(K(Z)\to P(Z)\) is injective, so is the map \(K_\sesi(Y)\to P_\sesi(Y)\).  Hence the map \(K_\sesi\to P_\sesi\) is a monomorphism, forcing \(\Tor^\Nattrafo_1(\Nattrafo_\sesi,M)=0\) by~\eqref{eq:Tor_sequence}.
  \end{proof}

  To finish the proof of the implication (4)\(\Longrightarrow\)(3) in Theorem~\ref{the:projective_module_Nattrafo}, it remains to check that \(M_\sesi(Y)\) is free for all~\(Y\) if~\(M\) is exact and \(M(Y)\) is free for all~\(Y\).  We use Lemma~\ref{lem:gen_M_nil} once again to describe \(M_\sesi(Y)\) as the range of a canonical element in \(\Nattrafo_*(Y,Z)\) for a suitable~\(Z\).  Thus \(M_\sesi(Y)\) is isomorphic to a subgroup of \(M(Z)\), which is a free group by assumption.  Hence \(M_\sesi(Y)\) is free as well.
\end{proof}

\section{Homological algebra in \texorpdfstring{$\KKcat(X)$}{KK(X)}}
\label{sec:filtered_K}

Let~\(X\) be a sober topological space.  We are going to apply to \(\KKcat(X)\) the general machinery for doing homological algebra in triangulated categories discussed in~\cite{Meyer-Nest:Homology_in_KK}.  This theory goes back to the work on relative homological algebra by Samuel Eilenberg and John Coleman Moore (\cite{Eilenberg-Moore:Foundations}), which was carried over to the setting of triangulated categories by Daniel Christensen~\cite{Christensen:Ideals} and Apostolos Beligiannis~\cite{Beligiannis:Relative}.

\subsection{An ideal in \texorpdfstring{$\KKcat(X)$}{KK(X)}}
\label{sec:ideal}

Our starting point is a \emph{rough idea} of the invariant we want to use.  This rough idea is expressed by a homological ideal in the triangulated category.  The ideal~\(\Ideal\) in \(\KKcat(X)\) relevant for us is defined by
\begin{multline}
  \label{eq:def_Ideal}
  \Ideal(A,B) \defeq \bigl\{f\in\KK(X;A,B) \bigm|\\
  \text{\(f_*\colon \K_*\bigl(A(Y)\bigr)\to
    \K_*\bigl(B(Y)\bigr)\) vanishes
    for all \(Y\in\Loclo(X)\)}\bigr\}.
\end{multline}
It makes no difference if we use \(\Loclo(X)\) or~\(\Loclo(X)^*\) here.

We claim that~\(\Ideal\) is a \emph{homological ideal} in the triangulated category \(\KKcat(X)\); that is, it is the kernel (on morphisms) of a \emph{stable homological functor} from \(\KKcat(X)\) to some stable Abelian category; stability means that the functor intertwines the suspension automorphism on \(\KKcat(X)\) with a given suspension automorphism on the target Abelian category.

Our starting point is a bare form of filtrated \(\K\)\nb-theory.  Recall the functors
\[
\FK_Y\colon\KKcat(X)\to \Ab^{\Z/2},\qquad
A\mapsto \K_*\bigl(A(Y)\bigr)
\]
for \(Y\in\Loclo(X)\) from Definition~\ref{def:FK_Y} and let
\[
F\defeq (\FK_Y)_{Y\in\Loclo(X)}\colon
\KKcat(X)\to \prod_{Y\in\Loclo(X)^*} \Ab^{\Z/2},
\qquad
A\mapsto \Bigl(\K_*\bigl(A(Y)\bigr)\Bigr)_{Y\in\Loclo(X)^*}.
\]
The target category \(\prod_{Y\in\Loclo(X)^*} \Ab^{\Z/2}\) of~\(F\) is Abelian and carries an obvious suspension functor that shifts the \(\Z/2\)\nb-grading.  The functor~\(F\) is a stable homological functor, that is, it intertwines the suspension automorphisms and maps exact triangles to long exact sequences.  By definition,
\begin{equation}
  \label{eq:FK_as_kernel}
  \Ideal =  \bigcap_{Y\in\Loclo(X)^*} \ker \FK_Y = \ker F,
\end{equation}
that is, \(f\in\Ideal(A,B)\) if and only if \(F(f)=0\).  Hence~\(\Ideal\) is a homological ideal with defining functor~\(F\).

We also have \(\Ideal=\ker \FK\) with~\(\FK\) as in Definition~\ref{def:filtrated_K}: the two functors \(F\) and~\(\FK\) only differ through their target categories.  For the time being, we pretend that we do not yet know anything about filtrated \(\K\)\nb-theory beyond the ideal~\(\Ideal\) it defines.  The general machinery will \emph{automatically} lead us to the functor~\(\FK\).

As explained in~\cite{Meyer-Nest:Homology_in_KK}, the homological ideal~\(\Ideal\) yields various notions of homological algebra.  The following descriptions of these notions follow from \cite{Meyer-Nest:Homology_in_KK}*{Lemmas 3.2 and 3.9, Definition 3.21}.
\begin{itemize}
\item A morphism \(f\in \KK_*(X;A,B)\) is
  \begin{itemize}
  \item \textit{\(\Ideal\)\nb-epic} if the induced maps \(\K_*\bigl(A(Y)\bigr) \to \K_*\bigl(B(Y)\bigr)\) are surjective for all \(Y\in\Loclo(X)\);

  \item \emph{\(\Ideal\)\nb-monic} if the induced maps \(\K_*\bigl(A(Y)\bigr) \to \K_*\bigl(B(Y)\bigr)\) are injective for all \(Y\in\Loclo(X)\);

  \item an \emph{\(\Ideal\)\nb-equivalence} if the induced maps \(\K_*\bigl(A(Y)\bigr) \to \K_*\bigl(B(Y)\bigr)\) are bijective for all \(Y\in\Loclo(X)\).

  \end{itemize}

\item A homological functor \(F\colon \KKcat(X)\to\Cat\) to some Abelian category~\(\Cat\) is \emph{\(\Ideal\)\nb-exact} if \(F(f)=0\) for all \(f\in\Ideal\); equivalently, \(F\) maps \(\Ideal\)\nb-epimorphisms to epimorphisms or~\(F\) maps \(\Ideal\)\nb-monomorphisms to monomorphisms.

\item An object \(A\inOb\KKcat(X)\) is
  \begin{itemize}
  \item \emph{\(\Ideal\)\nb-contractible} if \(\K_*\bigl(A(Y)\bigr)=0\) for all \(Y\in\Loclo(X)\);

  \item \emph{\(\Ideal\)\nb-projective} if the functor \(\KK_*(X;A,\blank)\) is \(\Ideal\)\nb-exact; equivalently, \(\Ideal(A,B)=0\) for all \(B\inOb\KKcat(X)\), or: any \(\Ideal\)\nb-epimorphism \(B\to A\) splits (see~\cite{Meyer-Nest:Homology_in_KK} for more equivalent characterisations).

  \end{itemize}

\item A chain complex
  \[
  \dotsb \to
   A_{n+1} \xrightarrow{\delta_{n+1}}
   A_n \xrightarrow{\delta_n}
   A_{n-1} \xrightarrow{\delta_{n-1}}
   A_{n-2} \to
   \dotsb
  \]
  in \(\KKcat(X)\)---that is, \(A_n\inOb\KKcat(X)\) and \(\delta_n\in\KK(X;A_n,A_{n-1})\) for all \(n\in\Z\), subject to the condition \(\delta_{n-1}\circ\delta_n=0\)---is \emph{\(\Ideal\)\nb-exact} (in some degree~\(n\)) if the induced chain complexes of \(\Z/2\)-graded Abelian groups
  \[
  \dotsb \to
  \K_*\bigl(A_{n+1}(Y)\bigr) \xrightarrow{(\delta_{n+1})_*}
  \K_*\bigl(A_{n}(Y)\bigr) \xrightarrow{(\delta_{n})_*}
  \K_*\bigl(A_{n-1}(Y)\bigr) \to
  \dotsb
  \]
  are exact (in degree~\(n\)) for all \(Y\in\Loclo(X)\).

\item An \emph{\(\Ideal\)\nb-projective resolution} of \(A\inOb\KKcat(X)\) is an \(\Ideal\)\nb-exact chain complex
  \[
  \dotsb \to
   P_2 \xrightarrow{\delta_2}
   P_1 \xrightarrow{\delta_1}
   P_0 \xrightarrow{\delta_0}
   A \to 0 \to
   \dotsb
  \]
  with \(\Ideal\)\nb-projective entries~\(P_n\) for all \(n\in\N\).
\end{itemize}

We shall soon see that there are \emph{enough \(\Ideal\)\nb-projective objects} in the sense that any object of \(\KKcat(X)\) has an \(\Ideal\)\nb-projective resolution.  Such resolutions are unique up to chain homotopy equivalence once they exist.

We use projective resolutions to define \emph{derived functors} (see \cite{Meyer-Nest:Homology_in_KK}*{Definition 3.27}): just apply the functor to be derived to an \(\Ideal\)\nb-projective resolution and take homology.  In particular, this yields extension groups \(\Ext^n_\Ideal(A,B)\) for all \(A,B\inOb\KKcat(X)\).  Unlike in usual homological algebra, \(\Ext^0_\Ideal(A,B)\) may differ from the morphism space in \(\KKcat(X)\), compare the exact sequence~(4.8) in~\cite{Meyer:Homology_in_KK_II}.

\subsection{Enough projective objects}
\label{sec:enough_projectives}

A strategy to find enough projective objects is outlined in \cite{Meyer-Nest:Homology_in_KK}*{\S3.6}.  The idea is to study the left adjoint functor \(\FK_Y^\lad\) of \(\FK_Y\); this is defined on \(P\inOb\Ab^{\Z/2}\) if there is \(\FK_Y^\lad(P)\inOb\KKcat(X)\) and a natural isomorphism
\begin{equation}
  \label{eq:FK_adjoint}
  \Hom\bigl(P,\FK_Y(B)\bigr) \cong \KK(X;\FK_Y^\lad(P),B)
\end{equation}
for all \(B\inOb\KKcat(X)\).  Notice that \(\FK_Y^\lad\) need not be defined for all~\(P\).

Objects of the form \(\FK_Y^\lad(P)\) are automatically \(\Ideal\)\nb-projective because the functor \(\KK(X;\FK_Y^\lad(P),\blank)\) factors through~\(\FK_Y\) by~\eqref{eq:FK_adjoint} and vanishes on~\(\Ideal\) by~\eqref{eq:FK_as_kernel}.

The simplest case to look for \(\FK^\lad_Y(P)\) is \(P=\Z[0]\) (this means~\(\Z\) in degree~\(0\)).  The defining property of \(\FK^\lad_Y(\Z[0])\) is a natural isomorphism
\[
\KK(X;\FK_Y^\lad(\Z[0]),B)
\cong \Hom\bigl(\Z[0],\FK_Y(B)\bigr)
\cong \FK_{Y,0}(B)
= \K_0\bigl(B(Y)\bigr).
\]
In other words, \(\FK_Y^\lad(\Z[0])\) must represent the covariant functor~\(\FK_Y\).  Theorem~\ref{the:representability} provides such representing objects, and yields the following:

\begin{proposition}
  \label{pro:adjoint_FK_Y}
  For any \(Y\in\Loclo(X)\), the adjoint functor \(\FK_Y^\lad\) is defined on a \(\Z/2\)-graded Abelian group \(G=G_0\oplus G_1\) if \(G_0\) and~\(G_1\) are free and countable.  More precisely,
  \[
  \FK_Y^\lad\left(\bigoplus_{i\in I}
    \Z[\varepsilon_i]\right)
  = \bigoplus_{i\in I} \Repr_Y[\varepsilon_i],
  \]
  where~\(I\) is a countable set and \(\varepsilon_i\in\Z/2\) for all \(i\in I\).
\end{proposition}

\begin{proof}
  We have just observed that \(\FK_Y^\lad(\Z[0])=\Repr_Y\).  Since~\(\FK_Y\) is stable, this implies \(\FK_Y^\lad(\Z[1])=\Repr_Y[1]\).  It is a general feature of left adjoint functors that they commute with direct sums.  Since countable direct sums exist in \(\KKcat(X)\), we get the existence of~\(\FK_Y^\lad\) on any free countable \(\Z/2\)-graded Abelian group.
\end{proof}

\begin{corollary}
  \label{cor:enough_projectives}
  There are enough \(\Ideal\)\nb-projective objects in \(\KKcat(X)\), and the class of \(\Ideal\)\nb-projective objects in \(\KKcat(X)\) is generated by the objects~\(\Repr_Y\) for \(Y\in\Loclo(X)^*\).  More precisely, any \(\Ideal\)\nb-projective objects is a retract of a direct sum of suspensions of these objects.
\end{corollary}

\begin{proof}
  This follows from Proposition~\ref{pro:adjoint_FK_Y} and \cite{Meyer-Nest:Homology_in_KK}*{Proposition 3.37}.
\end{proof}

Often we do not need retracts here, that is, any \(\Ideal\)\nb-projective object is a direct sum of suspensions of~\(\Repr_Y\) for \(Y\in\Loclo(X)^*\); for the totally ordered spaces studied in~\S\ref{sec:example_Nattrafo}, this follows from Theorem~\ref{the:projective_module_Nattrafo}.

Since our ideal~\(\Ideal\) is compatible with countable direct sums, the \(\Ideal\)\nb-contractible objects form a localising subcategory of \(\KKcat(X)\), that is, they form a class~\(\Null_\Ideal\) of objects that is closed under countable direct sums, retracts, isomorphism, exact triangles, and suspensions.  Furthermore, \(\Null_\Ideal\) is the complement of the localising subcategory that is generated by the \(\Ideal\)\nb-projective objects.  These two subcategories contain much less information than the ideal itself.  Roughly speaking, they will be the same for any reasonable choice of invariant on \(\KKcat(X)\) of \(\K\)\nb-theoretic nature.

\begin{proposition}
  \label{pro:bootstrap_proj_Ideal}
  The localising subcategory that is generated by the \(\Ideal\)\nb-projective objects is the bootstrap category \(\Bootstrap(X)\).  It consists of all objects of \(\KKcat(X)\) that are \(\KK(X)\)-equivalent to a tight, nuclear, purely infinite, stable, separable \(\Cst\)\nb-algebra over~\(X\) whose simple subquotients belong to the bootstrap category \(\Bootstrap\subseteq\KKcat\).
\end{proposition}

\begin{proof}
  By definition, \(\Bootstrap(X)\) is the localising subcategory of \(\KKcat(X)\) that is generated by the objects \(i_x(\C)\) for \(x\in X\), see~\cite{Meyer-Nest:Bootstrap}.  These generators are \(\Ideal\)\nb-projective because they represent the functors \(\FK_{U_x}\), compare the proof of the Representability Theorem~\ref{the:representability}.  The proof of this theorem also shows that the representing objects~\(\Repr_Y\) belong to the triangulated subcategory of \(\KKcat(X)\) generated by \(\Repr_{U_x}\) for \(x\in X\) and hence to \(\Bootstrap(X)\).  Now Corollary~\ref{cor:enough_projectives} shows that all \(\Ideal\)\nb-projective objects belong to \(\Bootstrap(X)\).  Hence the localising subcategory they generate is contained in the bootstrap class.

  Conversely, since the generators of the bootstrap class \(i_x(\C)\) are \(\Ideal\)\nb-projective, the localising subcategory generated by the \(\Ideal\)\nb-projective objects must contain the whole bootstrap class.  This yields the first statement.  The second one is contained in \cite{Meyer-Nest:Bootstrap}*{Corollary~5.5}.
\end{proof}

\subsection{The universality of filtrated \texorpdfstring{$\K$}{K}-theory}
\label{sec:FK_universal}

The next step in the general programme is to determine the \emph{universal} defining functor for~\(\Ideal\).  This functor is characterised by the universal property that it is \(\Ideal\)\nb-exact and stable homological and that any \(\Ideal\)\nb-exact homological functor on \(\KKcat(X)\) factors through it uniquely (up to natural isomorphism).

The advantage of using the universal functor is that it describes \(\Ideal\)\nb-projective resolutions and the associated \(\Ideal\)\nb-derived functors in \(\KKcat(X)\) by projective resolutions and derived functors in its target Abelian category.  This is the crucial step to compute these derived functors.

In the presence of enough projective objects, \cite{Meyer-Nest:Homology_in_KK}*{Theorem 3.39} characterises the universal functor by an adjointness property.  In our case, this yields:

\begin{theorem}
  \label{the:FK_universal}
  The filtrated \(\K\)\nb-theory functor \(\FK\colon \KKcat(X)\to \CMod{\Nattrafo}\) is the universal \(\Ideal\)\nb-exact stable homological functor; here \(\CMod{\Nattrafo}\) denotes the category of all countable graded \(\Nattrafo\)\nb-modules.
\end{theorem}

The ring of natural transformations \(\Nattrafo\) comes in automatically at this point.

\begin{proof}
  This is best explained as a special case of a general result on certain homological ideals.  Let~\(\Tri\) be any triangulated category with countable direct sums, and let~\(\Gen\) be an at most countable set of objects of~\(\Tri\).  Let~\(\Ideal_\Gen\) be the stable homological ideal defined by the functor
  \[
  F_\Gen\colon \Tri\to \prod_{G\in\Gen} \Ab^\Z,\qquad
  A \mapsto \bigl(\Tri_*(G,A)\bigr)_{G\in\Gen}.
  \]
  We assume that \(F_\Gen(A)\) is countable for all \(A\inOb\Tri\).

  We are dealing with the case where \(\Tri=\KKcat(X)\) and \(\Gen=\{\Repr_Y\mid Y\in\Loclo(X)^*\}\); Theorem~\ref{the:representability} identifies \(\Tri_*(\Repr_Y,A) = \KK_*(X;\Repr_Y,A)\cong \K_*\bigl(A(Y)\bigr) = \FK_Y(A)\) for all \(Y\in\Loclo(X)^*\), so that \(\Ideal_\Gen=\Ideal\) with~\(\Ideal\) as in~\eqref{eq:def_Ideal}.

  Viewing~\(\Gen\) as a full subcategory of~\(\Tri\), it becomes a \(\Z\)\nb-graded pre-additive category, so that we get a corresponding category \(\CMod{\Gen^\op}\) of countable graded \emph{right} modules.  We can enrich the functor~\(F_\Gen\) to a functor
  \[
  F_\Gen'\colon \Tri\to\CMod{\Gen^\op}
  \]
  because the composition in~\(\Tri\) provides maps
  \[
  \Tri_*(G',A)\otimes \Tri_*(G,G') \to \Tri_*(G,A)
  \]
  for all \(G,G'\in\Gen\), \(A\inOb\Tri\), which form a right \(\Gen\)\nb-module structure on \(\bigl(\Tri_*(G,A)\bigr)_{G\in\Gen}\).  We claim that the functor \(F_\Gen'\) is the universal \(\Ideal_\Gen\)\nb-exact functor.

  In the case at hand, our description of the natural transformations \(\FK_Y\Rightarrow\FK_Z\) in \S\ref{sec:representability} means that \(\CMod{\Gen^\op}=\CMod{\Nattrafo}\) and \(F_\Gen'=\FK\) is filtrated \(\K\)\nb-theory as defined in Definition~\ref{def:filtrated_K}.  Hence it suffices to establish the claim above to finish the proof of Theorem~\ref{the:FK_universal}.

  To do this, we check the conditions in \cite{Meyer-Nest:Homology_in_KK}*{Theorem 3.39}.  Idempotent morphisms in \(\KKcat(X)\) split because this happens in any triangulated category with countable direct sums (see~\cite{Neeman:Triangulated}).  Call \(F'_\Gen(G)= \Tri(\blank,G)\) for \(G\in\Gen\) the \emph{free} \(\Gen^\op\)-module on~\(G\).  Direct sums of free modules are projective, and any object of \(\CMod{\Gen^\op}\) is a quotient of a countable direct sum of free modules.  Hence \(\CMod{\Gen^\op}\) has enough projective objects.  Moreover,
  \[
  \Hom_{\Gen^\op}\bigl(F'_\Gen(G), F'_\Gen(A)\bigr)
  \cong F'_\Gen(A)(G)
  = \Tri(G,A)
  \]
  shows that the left adjoint~\(F^\lad\) of \(F\defeq F'_\Gen\) maps~\(F'_\Gen(G)\) to~\(G\inOb\Tri\).  Since the domain of~\(F^\lad\) is closed under suspensions, countable direct sums, and retracts, the adjoint is defined on all projective modules.  Furthermore, \(F\circ F^\lad(P) \cong P\) holds for free modules and hence for all projective modules~\(P\).  Having checked all the hypotheses of \cite{Meyer-Nest:Homology_in_KK}*{Theorem 3.39}, we can conclude that~\(F'_\Gen\) is indeed universal.
\end{proof}

Since \(\FK\colon \KKcat(X)\to\CMod{\Nattrafo}\) is universal, \cite{Meyer-Nest:Homology_in_KK}*{Theorem 3.41} now tells us, roughly speaking, that homological algebra in \(\KKcat(X)\) with respect to~\(\Ideal\) is equivalent to homological algebra in the Abelian category \(\CMod{\Nattrafo}\):
\begin{itemize}
\item An object~\(A\) of \(\KKcat(X)\) is \(\Ideal\)\nb-projective if and only if \(\FK(A)\in\CMod{\Nattrafo}\) is projective and
  \[
  \KK_*(X;A,B) \cong \Hom_\Nattrafo\bigl(\FK(A),\FK(B)\bigr)
  \]
  for all \(B\inOb\KKcat(X)\).

  Another equivalent condition is that \(\FK(A)\in\CMod{\Nattrafo}\) is projective and~\(A\) belongs to the localising subcategory generated by the \(\Ideal\)\nb-projective objects; the latter agrees with the bootstrap class by Proposition~\ref{pro:bootstrap_proj_Ideal}.

\item The functor \(\FK\) and its partially defined left adjoint~\(\FK^\lad\) restrict to an equivalence of categories between the subcategories of \(\Ideal\)\nb-projective objects in \(\KKcat(X)\) and of projective objects in \(\CMod{\Nattrafo}\).

\item For any \(A\inOb\KKcat(X)\), the functors \(\FK\) and~\(\FK^\lad\) induce bijections between isomorphism classes of \(\Ideal\)\nb-projective resolutions of~\(A\) and isomorphism classes of projective resolutions of \(\FK(A)\) in \(\CMod{\Nattrafo}\).  That is, a projective resolution in \(\CMod{\Nattrafo}\) lifts to a unique \(\Ideal\)\nb-projective resolution in \(\KKcat(X)\).  This provides the ``geometric resolutions'' that are used in connection with the usual Universal Coefficient Theorem for \(\KK\).

\item For all \(n\in\N\), there is a natural isomorphism
  \[
  \Ext_\Ideal^n(A,B)
  \cong \Ext^n_\Nattrafo\bigl(\FK(A),\FK(B)\bigr),
  \]
  where the right hand side denotes extension groups in the Abelian category \(\CMod{\Nattrafo}\).

\item For any homological functor \(G\colon \KKcat(X)\to\Cat\), there is a unique right-exact functor \(\bar{G}\colon \CMod{\Nattrafo}\to\Cat\) with \(\bar{G}\circ\FK(P)= G(P)\) for all \(\Ideal\)\nb-projective~\(P\).  The left derived functors of~\(G\) with respect to~\(\Ideal\) are \(\Left_n\bar{G}\circ \FK\) for \(n\in\N\), where \(\Left_n\bar{G}\colon \CMod{\Nattrafo}\to\Cat\) denotes the \(n\)th left derived functor of~\(\bar{G}\).
\end{itemize}

\subsection{The Universal Coefficient Theorem}
\label{sec:UCT_conditional}

In the general theory, the next step is to construct a spectral sequence whose \(E^2\)\nb-term involves the extension groups \(\Ext_\Ideal^n(A[m],B)\); it converges---in favourable cases---to \(\KK_*(X;A,B)\).  This spectral sequence is constructed in \cites{Christensen:Ideals, Meyer:Homology_in_KK_II}.  Since we aim for an exact sequence, not for a spectral sequence, we only need the special case considered in \cite{Meyer-Nest:Homology_in_KK}*{Theorem 4.4}.  This provides the Universal Coefficient Theorem we want under the assumption that \(\FK(A)\) has a projective resolution of length~\(1\) in \(\CMod{\Nattrafo}\):

\begin{theorem}
  \label{the:UCT_conditional}
  Let \(A,B\inOb\KKcat(X)\).  Suppose that \(\FK(A)\inOb\CMod{\Nattrafo}\) has a projective resolution of length~\(1\) and that \(A\inOb\Bootstrap(X)\).  Then there are natural short exact sequences
  \[
  \Ext^1_\Nattrafo\bigl(\FK(A)[j+1],\FK(B)\bigr)
  \into \KK_j(X;A,B)
  \prto \Hom_\Nattrafo\bigl(\FK(A)[j],\FK(B)\bigr)
  \]
  for \(j\in\Z/2\), where \(\Hom_\Nattrafo\) and \(\Ext^1_\Nattrafo\) denote the morphism and extension groups in the Abelian category \(\CMod{\Nattrafo}\) and \([j]\) and \([j+1]\) denote degree shifts.
\end{theorem}

The bootstrap class appears here because of Proposition~\ref{pro:bootstrap_proj_Ideal}, which identifies it with the localising subcategory generated by the \(\Ideal\)\nb-projective objects.

\begin{corollary}
  \label{cor:lift_iso}
  Let \(A,B\inOb\Bootstrap(X)\) and suppose that both \(\FK(A)\) and \(\FK(B)\) have projective resolutions of length~\(1\) in \(\CMod{\Nattrafo}\).  Then any morphism \(\FK(A)\to\FK(B)\) in \(\CMod{\Nattrafo}\) lifts to an element in \(\KK_0(X;A,B)\), and an isomorphism \(\FK(A)\cong\FK(B)\) lifts to an isomorphism in \(\Bootstrap(X)\).
\end{corollary}

\begin{proof}
  The lifting of a homomorphism follows from Theorem~\ref{the:UCT_conditional}.  Given an isomorphism \(f\colon \FK(A)\to\FK(B)\), we can lift \(f\) and~\(f^{-1}\) to elements \(\alpha\) and~\(\beta\) of \(\KK_0(X;A,B)\) and \(\KK_0(X;B,A)\), respectively.  Since \(\beta\circ\alpha\) lifts the identity map on \(\FK(A)\), the difference \(\ID-\beta\circ\alpha\) belongs to \(\Ext^1_\Nattrafo\bigl(\FK(A)[j+1],\FK(A)\bigr)\).  The latter is a nilpotent ideal in \(\KK(X;A,A)\) because of the naturality of the exact sequence in Theorem~\ref{the:UCT_conditional}.  Hence \((\ID-\beta\alpha)^2=0\), so that \(\beta\circ\alpha\) is invertible.  The same argument shows that \(\alpha\circ\beta\) is invertible, so that~\(\alpha\) is invertible.
\end{proof}

This corollary is what is needed for the classification programme, and it depends on resolutions having length~\(1\).  Conversely, if there is~\(A\) for which \(\FK(A)\) has no projective resolution of length~\(1\), then it is likely that there exist non-isomorphic \(B,D\inOb\Bootstrap(X)\) with \(\FK(B)\cong \FK(D)\).  The following theorem provides such a counterexample, but under a stronger assumption.

\begin{theorem}
  \label{the:long-resolution_problem}
  Let~\(\Ideal\) be a homological ideal in a triangulated category~\(\Tri\) with enough \(\Ideal\)\nb-projective objects.  Let \(F\colon\Tri\to\Abel_\Ideal\Tri\) be a universal \(\Ideal\)\nb-exact stable homological functor.  Suppose that \(\Ideal^2\neq0\).  Then there exist non-isomorphic objects \(B,D\inOb\Tri\) for which \(F(B)\cong F(D)\) in~\(\Abel_\Ideal\Tri\).
\end{theorem}

\begin{proof}
  Since \(\Ideal^2\neq0\), there is \(A\inOb\Tri\) with \(\Ideal^2(A,\blank)\neq0\), that is, \(A\) is not \(\Ideal^2\)\nb-projective.  The ideal~\(\Ideal^2\) has enough projective objects as well, so that there is an exact triangle
  \[
  \Sigma N_2 \xrightarrow{\gamma_2}
  \tilde{A}_2 \xrightarrow{\alpha_2}
  A \xrightarrow{\iota_2} N_2
  \]
  with \(\iota_2\in\Ideal^2\) and an \(\Ideal^2\)\nb-projective object~\(\tilde{A}_2\) (this is part of the phantom castle constructed in~\cite{Meyer:Homology_in_KK_II}, where the same notation is used).

  Since \(\iota_2\in\Ideal\), this triangle is \(\Ideal\)\nb-exact and hence provides an extension
  \[
  F(N_2)[1] \into F(\tilde{A}_2) \prto F(A)
  \]
  in \(\Abel_\Ideal\Tri\).  Even more, this extension splits because \(\iota_2\in\Ideal^2\).  This follows because the canonical map
  \[
  \Ideal(A,N_2) \to \Ext^1_\Ideal(A,N_2[1])
  \]
  implicitly used above factors through \(\Ideal/\Ideal^2\) and hence annihilates~\(\iota_2\) (see \cite{Meyer:Homology_in_KK_II}*{Equation~(4.9)}).  As a result, \(F(\tilde{A}_2) \cong F(A)\oplus F(N_2)[1]\).

  But~\(\tilde{A}_2\) cannot be isomorphic to \(A\oplus N_2[1]\).  If this were the case, then~\(A\) would be \(\Ideal^2\)\nb-projective, as a retract of the \(\Ideal^2\)\nb-projective object~\(\tilde{A}_2\).  Then \(\Ideal^2(A,\blank)=0\), contradicting our choice of~\(A\).  Hence \(\tilde{A}_2 \not\cong A\oplus N_2[1]\).
\end{proof}

If \(\Ideal^2=0\), then the ABC spectral sequence constructed in~\cite{Meyer:Homology_in_KK_II} degenerates at the third stage, that is, \(E^3=E^\infty\).  But \(E^2\) and~\(E^3\) differ unless projective resolutions have length~\(1\).  Hence the vanishing of~\(\Ideal^2\) is probably not sufficient for isomorphisms on the invariant to lift because the boundary map~\(d^2\) on the second stage of the ABC spectral sequence may provide further obstructions.

Whether or not filtrated \(\K\)\nb-theory gives rise to projective resolutions of length~\(1\) depends on the space in question: we will find positive and negative cases below.  Before we turn to examples, we discuss another important issue: does filtrated \(\K\)\nb-theory exhaust all of \(\CMod{\Nattrafo}\)?  This is definitely not the case because of the additional exactness conditions that hold for objects of the form \(\FK(A)\).  The following result is not optimal but sufficient for our purposes.

\begin{theorem}
  \label{the:FK_exhausts_conditional}
  Let \(G\inOb\CMod{\Nattrafo}\) have a projective resolution of length~\(1\).  Then there is \(A\inOb\Bootstrap(X)\) with \(\FK(A)\cong G\), and this object is unique up to isomorphism in \(\Bootstrap(X)\).
\end{theorem}

\begin{proof}
  Any projective resolution of length~\(1\) in \(\CMod{\Nattrafo}\) is isomorphic to one of the form
  \[
  \dotsb \to 0 \to \FK(P_1)
  \xrightarrow{\FK(f)} \FK(P_0)\to G
  \]
  for suitable \(\Ideal\)\nb-projective objects \(P_1,P_0\inOb\KKcat(X)\) and some \(f\in\KK_0(X;P_1,P_0)\).  Here we use that~\(\FK\) restricts to an equivalence of categories between the subcategories of \(\Ideal\)\nb-projective objects of \(\KKcat(X)\) and of projective objects of \(\CMod{\Nattrafo}\) by the first paragraph of \cite{Meyer-Nest:Homology_in_KK}*{Theorem 3.41}.

  We may embed the morphism~\(f\) in an exact triangle
  \[
  \Sigma A \xrightarrow{h} P_1 \xrightarrow{f} P_0 \xrightarrow{g} A.
  \]
  Since \(\FK(f)\) is injective, the map~\(f\) is \(\Ideal\)\nb-monic; thus \(g\) is \(\Ideal\)\nb-epic and \(h\in\Ideal\).  Therefore, the long exact sequence for~\(\FK\) applied to the above triangle degenerates to a short exact sequence
  \[
  \FK(P_1) \into \FK(P_0) \prto \FK(A).
  \]
  This yields \(\FK(A)\cong G\) as desired.  The uniqueness of~\(A\) is already contained in Corollary~\ref{cor:lift_iso}.
\end{proof}

It remains to understand which objects of the category \(\CMod{\Nattrafo}\) have a projective resolution of length~\(1\).

\subsection{Resolutions of length~\texorpdfstring{$1$}{1} in the totally ordered case}
\label{sec:resolutions_length}

We return to the example of the space \(X=\{1,\dotsc,n\}\) totally ordered by~\(\le\) studied in~\S\ref{sec:example_Nattrafo}.  Let~\(\Nattrafo\) be the graded pre-additive category of natural transformations described in \S\ref{sec:example_Nattrafo}, and let \(\Cat=\CMod{\Nattrafo}\) be the Abelian category of \(\Nattrafo\)\nb-modules.  The following theorem characterises \(\Nattrafo\)\nb-modules with projective resolutions of length~\(1\):

\begin{theorem}
  \label{the:length_one_example}
  Let \(M\inOb\Cat\).  The following assertions are equivalent:
  \begin{enumerate}[label=\textup{(\roman{*})}]
  \item\label{loe_1} \(M=\FK_*(A)\) for some \(A\inOb\KKcat(X)\);

  \item\label{loe_2} \(M\) is exact in the sense of Definition~\textup{\ref{def:module_exact}};

  \item\label{loe_5} \(\Tor^\Nattrafo_i(\Nattrafo_\sesi,M)=0\) for \(i=1,2\);

  \item\label{loe_3} \(M\) has a free resolution of length~\(1\) in~\(\Cat\);

  \item\label{loe_6} \(M\) has a projective resolution of length~\(1\) in~\(\Cat\);

  \item\label{loe_4} \(M\) has a projective resolution of finite length in~\(\Cat\).
  \end{enumerate}
\end{theorem}

\begin{proof}
  The exact sequence~\eqref{eq:six-term_intro} shows that~\ref{loe_1} implies~\ref{loe_2}.  Theorem~\ref{the:FK_exhausts_conditional} contains the implication \ref{loe_6}\(\Longrightarrow\)\ref{loe_1}, and the implications \ref{loe_3}\(\Longrightarrow\)\ref{loe_6}\(\Longrightarrow\)\ref{loe_4} are trivial.  We will show \ref{loe_2}\(\Longrightarrow\)\ref{loe_5}\(\Longrightarrow\)\ref{loe_3} and \ref{loe_4}\(\Longrightarrow\)\ref{loe_2}, and this will establish the theorem.

  First we show that~\ref{loe_4} implies~\ref{loe_2}.  Let \(0\to P_m\to \dotsb\to P_0 \to M\) be a projective resolution of finite length.  By a standard ``stabilisation'' trick, we can turn this into a free resolution of the same length.  Let
  \[
  Z_j=\ker(P_j\to P_{j-1}) \cong \range (P_{j+1}\to P_j).
  \]
  Thus \(Z_m=0\), \(P_0/Z_0\cong M\), and we have exact sequences \(Z_j\into P_j \prto Z_{j-1}\) because our chain complex is exact.  Since \(Z_m=0\), the exactness of the projective modules~\(P_m\) and Proposition~\ref{pro:exact_modules} show recursively that~\(Z_j\) is exact for \(j=m-1,m-2,\dotsc,0\), so that~\(M\) is exact.  Thus~\ref{loe_4} implies~\ref{loe_2}.

  Now we prove \ref{loe_2}\(\Longrightarrow\)\ref{loe_5}\(\Longrightarrow\)\ref{loe_3}.  Let~\(P\) be a countable free module for which there is an epimorphism \(\pi\colon P\prto M\), and let \(K\defeq \ker \pi\).  We have an extension of \(\Nattrafo\)-modules \(K\into P\prto M\).  Proposition~\ref{pro:exact_modules} shows that~\(K\) is exact because \(P\) and~\(M\) are exact.  Furthermore, \(\Tor_{i+1}(\Nattrafo_\sesi,M) \cong \Tor_i(\Nattrafo_\sesi,K)\) for all \(i\ge1\) because~\(P\) is projective.  Lemma~\ref{lem:Tor_technical} applied to \(M\) and~\(K\) yields \(\Tor_i(\Nattrafo_\sesi,M)=0\) for \(i=1,2\) if~\(M\) is exact, that is, \ref{loe_2}\(\Longrightarrow\)\ref{loe_5}.  Now assume~\ref{loe_5}.  The argument above yields \(\Tor_1(\Nattrafo_\sesi,K)=0\).  Since~\(P\) is projective, the Abelian groups \(P(Y)\) are free for all \(Y\in\Loclo(X)\).  The exact sequence in~\eqref{eq:Tor_sequence} yields the same for~\(K(Y)\).  The criterion in Theorem \ref{the:projective_module_Nattrafo}.(3) shows that~\(K\) is projective.
\end{proof}

Now we combine the existence of projective resolutions of length~\(1\) with Theorem~\ref{the:UCT_conditional}, which still required this as a hypothesis:

\begin{theorem}
  \label{the:classify_ordered}
  Let~\(X\) be the topological space associated to a totally ordered finite set, and let \(A\) and~\(B\) be \(\Cst\)\nb-algebras over~\(X\).  If \(A\inOb\Bootstrap(X)\), then there is a natural short exact sequence
  \[
  \Ext^1_\Nattrafo\bigl(\FK(A)[1],\FK(B)\bigr)
  \into \KK_*(X;A,B)
  \prto \Hom_\Nattrafo\bigl(\FK(A),\FK(B)\bigr).
  \]
  In particular, any \(\Nattrafo\)\nb-module morphism \(\FK(A)\to\FK(B)\) lifts to an element in \(\KK_*(X;A,B)\).  If both \(A\) and~\(B\) belong to the bootstrap class \(\Bootstrap(X)\), then an isomorphism \(\FK(A)\cong\FK(B)\) lifts to a \(\KK\)\nb-equivalence \(A\simeq B\).
\end{theorem}

\begin{proof}
  Use Theorem~\ref{the:UCT_conditional} and Corollary~\ref{cor:lift_iso} together with the existence of projective resolutions of length~\(1\) ensured by Theorem~\ref{the:length_one_example}.
\end{proof}

\begin{theorem}
  \label{the:exhaust_ordered}
  Let~\(X\) be the topological space associated to a totally ordered finite set, and let \(A\) and~\(B\) be tight, purely infinite, stable, nuclear, separable \(\Cst\)\nb-algebras over~\(X\) whose simple subquotients belong to the bootstrap category.  Then an isomorphism \(\FK(A)\cong\FK(B)\) lifts to an \(X\)\nb-equivariant \Star{}isomorphism \(A\cong B\).

  Furthermore, any countable exact \(\Nattrafo\)\nb-modules is the filtrated \(\K\)\nb-module of some tight, purely infinite, stable, nuclear, separable \(\Cst\)\nb-algebra over~\(X\) with simple subquotients in the bootstrap category.
\end{theorem}

\begin{proof}
  A nuclear \(\Cst\)\nb-algebras over~\(X\) belongs to the bootstrap category \(\Bootstrap(X)\) if and only if its fibres belong to the non-equivariant bootstrap category~\(\Bootstrap\) (see \cite{Meyer-Nest:Bootstrap}*{Corollary 4.13}).  For a tight \(\Cst\)\nb-algebra over~\(X\), these fibres are the same as the simple subquotients.  It is also shown in~\cite{Meyer-Nest:Bootstrap}*{Corollary 5.5} that any object of \(\Bootstrap(X)\) is \(\KK(X)\)\nb-equivalent to a tight, nuclear, purely infinite, simple, separable \(\Cst\)\nb-algebra over~\(X\) whose simple subquotients belong to the bootstrap category~\(\Bootstrap\).  A deep classification result of Eberhard Kirchberg shows that any \(\KK(X)\)\nb-equivalence between such objects lifts to an \(X\)\nb-equivariant \Star{}homomorphism.  Now the first assertion follows from Theorem~\ref{the:classify_ordered}.  The second assertion also uses Theorem~\ref{the:FK_exhausts_conditional}.
\end{proof}

\section{A counterexample}
\label{sec:counterexample}

Now we let \(X\defeq \{1,2,3,4\}\) with the partial order \(1,2,3<4\) and no relation among \(1,2,3\).  Hence the open subsets of~\(X\) are
\[
\Open(X) = \bigl\{\emptyset, \{4\},\{1,4\},\{2,4\},\{3,4\},
\{1,2,4\},\{1,3,4\},\{2,3,4\},\{1,2,3,4\}\bigr\},
\]
that is, a non-empty subset is open if and only if it contains~\(4\).  The associated directed graph is
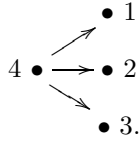
\[
\xymatrix@C-1em@R-1.5em{&\bullet\ 1\\4\
  \bullet\ar[r]\ar[dr]\ar[ur]&\bullet\ 2\\&\bullet\ 3.}
\]
We frequently denote subsets of~\(X\) simply by \(124\defeq \{1,2,4\}\), and so on.

A \(\Cst\)\nb-algebra over~\(X\) is a \(\Cst\)\nb-algebra~\(A\) with four distinguished ideals
\[
I_1\defeq A(14),\qquad
I_2\defeq A(24),\qquad
I_3\defeq A(34),\qquad
I_4\defeq A(4),
\]
such that \(I_1+I_2+I_3=A\) and \(I_i\cap I_j = I_4\) for all \(1\le i< j\le 3\) (see \cite{Meyer-Nest:Bootstrap}*{Lemma 2.35}).  Equivalently, the ideals \(I_j/I_4\) for \(j=1,2,3\) decompose \(A/I_4\) into a direct sum of three orthogonal ideals.  The other distinguished ideals are
\[
A(124) = I_1+I_2,\qquad
A(134) = I_1+I_3,\qquad
A(234) = I_2+I_3.
\]

Any subset of~\(X\) is locally closed.  But a \emph{connected} locally closed subset is either open or one of the singletons \(\{1\}\), \(\{2\}\), and~\(\{3\}\).  Hence the set of connected locally closed subsets is
\[
\Loclo(X)^* = \{4,14,24,34,124,134,234,1234,1,2,3\}.
\]

The order complex \(\Ch(X)\) is a graph with four vertices \(1,2,3,4\) and edges joining the first three to the last one:
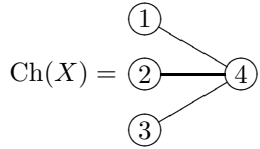
\[
\Ch(X) =
\begin{gathered}
  \entrymodifiers={=<1pc>[o][F-]}
  \xymatrix@R-1.5em{1\ar@{-}[dr]\\2\ar@{-}[r]&4\\3\ar@{-}[ur]}
\end{gathered}
\]
Both maps \(m,M\colon \Ch(X)\to X\) map the vertices to the corresponding points in~\(X\).  Whereas~\(M\) maps the interior of each edge to~\(4\), the map~\(m\) maps the interior of the edge \([j,4]\) to~\(j\) for \(j=1,2,3\).

Recall that the space of natural transformations \(\FK_Y\Rightarrow\FK_Z\) is given by
\[
\Nattrafo_*(Y,Z) \cong \K^*\bigl(S(Y,Z)\bigr),
\qquad
S(Y,Z) \defeq m^{-1}(Y)\cap M^{-1}(Z) \subseteq \Ch(X).
\]
It is straightforward to compute these \(\K\)\nb-theory groups, and the results are listed in Table~\ref{tab:nattrafo_counterexample}.
\begin{table}[htbp]
  \[
  \begin{array}{c|*{11}{l|}}
    Y\backslash Z&4&14&24&34&124&134&234&1234&1&2&3\\\hline
    4&\Z&\Z&\Z&\Z&\Z&\Z&\Z&\Z&0&0&0\\\hline
    14&0&\Z&0&0&\Z&\Z&0&\Z&\Z&0&0\\\hline
    24&0&0&\Z&0&\Z&0&\Z&\Z&0&\Z&0\\\hline
    34&0&0&0&\Z&0&\Z&\Z&\Z&0&0&\Z\\\hline
    124&\Z[1]&0&0&\Z[1]&\Z&0&0&\Z&\Z&\Z&0\\\hline
    134&\Z[1]&0&\Z[1]&0&0&\Z&0&\Z&\Z&0&\Z\\\hline
    234&\Z[1]&\Z[1]&0&0&0&0&\Z&\Z&0&\Z&\Z\\\hline
    1234&\Z[1]^2&\Z[1]&\Z[1]&\Z[1]&0&0&0&\Z&\Z&\Z&\Z\\\hline
    1&\Z[1]&0&\Z[1]&\Z[1]&0&0&\Z[1]&0&\Z&0&0\\\hline
    2&\Z[1]&\Z[1]&0&\Z[1]&0&\Z[1]&0&0&0&\Z&0\\\hline
    3&\Z[1]&\Z[1]&\Z[1]&0&\Z[1]&0&0&0&0&0&\Z\\\hline
  \end{array}
  \]
  \caption{The ring of natural transformations}
  \label{tab:nattrafo_counterexample}
\end{table}
Here the rows are labelled by~\(Y\), the columns by~\(Z\).  For instance, the entry~\(\Z\) at \((14,1)\) means that \(\Nattrafo_*(14,1)\cong\Z\).  The trivial \(1\)\nb-dimensional bundle over \(S(14,1)\) generates this group.  Hence Remark~\ref{rem:fundamental_classes_Repr} shows that the generator is the natural transformation that we get from the quotient map \(A(14)\prto A(1)\).  Similar arguments show that all the natural transformations of degree~\(0\) are induced by the familiar restriction and extension \Star{}homomorphisms for closed and open subsets.  Moreover, the odd natural transformations arise by composing these \Star{}homomorphisms with boundary maps in \(\K\)\nb-theory long exact sequences.  All relations that they satisfy are predicted by morphisms of extensions and exactness of the sequences~\eqref{eq:six-term_intro}.

The computations in \S\ref{sec:example_Nattrafo} were based on a description of indecomposable morphisms in the category \(\Nattrafo_*\).  For the space~\(X\) in question, these are the maps in the following diagram:
\begin{equation}
  \label{eq:Auslander-Reiten}
  \begin{gathered}
    \xymatrix@C+1em{
      & 14 \ar[r]^{i} \ar[dr]^{i}&
      124 \ar[dr]^{i}&& 1
      \ar[dr]|-\circ^{\delta}\\
      4 \ar[ur]^{i} \ar[r]^{i} \ar[dr]^{i}&
      24 \ar[ur]^{i} \ar[dr]^{i}&
      134 \ar[r]^{i}&
      1234 \ar[ur]^{r} \ar[r]^{r} \ar[dr]^{r}&
      2\ar[r]|-\circ^{\delta}&4\\
      & 34 \ar[r]^{i} \ar[ur]^{i}&
      234 \ar[ur]^{i}&& 3
      \ar[ur]|-\circ^{\delta}
    }
  \end{gathered}
\end{equation}
Here we write~\(i\) for the extension transformation for an open subset, \(r\) for the restriction transformation for a closed subset, and~\(\delta\) for boundary maps in \(\K\)\nb-theory long exact sequences.

The indecomposable morphisms in~\eqref{eq:Auslander-Reiten} provide a minimal set of generators for the graded ring \(\Nattrafo\).  To describe \(\Nattrafo\) completely, we list the relations.  These are generated by the following:
\begin{itemize}
\item the cube with vertices \(4\), \(14\), \dots, \(1234\) is a commuting diagram, that is, all the commuting squares involving arrows with label~\(i\) commute;

\item the following composite arrows vanish:
  \[
  124\xrightarrow{i} 1234\xrightarrow{r} 3,\qquad
  134\xrightarrow{i} 1234\xrightarrow{r} 2,\qquad
  234\xrightarrow{i} 1234\xrightarrow{r} 1,
  \]
  \[
  1\xrightarrow{\delta} 4\xrightarrow{i} 14,\qquad
  2\xrightarrow{\delta} 4\xrightarrow{i} 24,\qquad
  3\xrightarrow{\delta} 4\xrightarrow{i} 34;
  \]
\item the sum of the three maps \(1234\to 4\) via \(1\), \(2\), and~\(3\) vanishes.
\end{itemize}
These relations imply that the diagrams
\[
\xymatrix{
  124 \ar[r]^{r}\ar[d]^{r}\ar@{}[dr]|{-}&
  2\ar[d]|-\circ^{\delta}\\
  1\ar[r]|-\circ^{\delta}&4
}\qquad
\xymatrix{
  134 \ar[r]^{r}\ar[d]^{r}\ar@{}[dr]|{-}&
  3\ar[d]|-\circ^{\delta}\\
  1\ar[r]|-\circ^{\delta}&4
}\qquad
\xymatrix{
  234 \ar[r]^{r}\ar[d]^{r}\ar@{}[dr]|{-}&
  2\ar[d]|-\circ^{\delta}\\
  3\ar[r]|-\circ^{\delta}&4
}
\]
anti-commute and that the composite of two odd maps vanishes.  It is routine to check that the universal pre-additive category with these generators and relations is given by the groups listed in Table~\ref{tab:nattrafo_counterexample}.

Define \(\Nattrafo_\nil\) and~\(\Nattrafo_\sesi\) as in Definition~\ref{def:Nattrafo_nil}: \(\Nattrafo_\nil\) is the linear span of the groups \(\Nattrafo_*(Y,Z)\) with \(Y\neq Z\) and \(\Nattrafo_\sesi\) is spanned by the groups \(\Nattrafo_*(Y,Y)\).  Then \(\Nattrafo_\nil\) is a nilpotent ideal in~\(\Nattrafo\) and \(\Nattrafo_\sesi\cong \Z^{\Loclo(X)^*}\) is a semi-simple ring.  Thus~\(\Nattrafo_\nil\) is the maximal nilpotent ideal in~\(\Nattrafo\) and we have a semi-direct product decomposition \(\Nattrafo \cong \Nattrafo_\nil \rtimes \Nattrafo_\sesi\) as in Lemma~\ref{lem:Nattrafo_nil}.

The next task is to describe the submodule \(M'\defeq \Nattrafo_\nil\cdot M\subseteq M\) for an exact \(\Nattrafo\)\nb-module~\(M\).  The following computations are done as in the proof of Lemma~\ref{lem:gen_M_nil}, using \eqref{eq:range_sum} and that the morphisms in~\eqref{eq:Auslander-Reiten} generate \(\Nattrafo\).
\[
M'(14) = \range \bigl(i_4^{14}\colon M(4)\to M(14)\bigr)
= \ker \bigl(r_{14}^1\colon M(14)\to M(1)\bigr),
\]
and symmetrically for \(24\) and~\(34\);
\begin{multline*}
  M'(124) = \range \bigl(i_{14}^{124}\colon M(14)\to M(124)\bigr)
  + \bigl(i_{24}^{124}\colon M(24)\to M(124)\bigr)
  \\= \ker \bigl(\delta_{124}^4\colon M(124)\to M(4)\bigr),
\end{multline*}
where~\(\delta_{124}^4\) denotes a generator of \(\Nattrafo_1(124,4)\cong\Z\); symmetry provides \(M'(134)\) and \(M'(234)\).  We have
\[
M'(1) = \range \bigl(r_{1234}^1\colon M(1234)\to M(1)\bigr)
= \ker \bigl(\delta_1^{234}\colon M(1)\to M(234)\bigr),
\]
and symmetrically for \(2\) and~\(3\), and
\[
M'(4) = \sum_{j=1}^3
\range \bigl(\delta_j^4\colon M(j)\to M(4)\bigr)
= \ker \bigl(i_4^{1234}\colon M(4)\to M(1234)\bigr).
\]
But something goes wrong with \(M'(1234)\).  Equation~\eqref{eq:range_sum} yields
\begin{multline*}
  \range \bigl(i_{124}^{1234}\colon M(124)\to M(1234)\bigr)
  + \bigl(i_{134}^{1234}\colon M(134)\to M(1234)\bigr)
  \\= \ker \bigl(\delta_{1234}^{14}\colon M(1234)\to M(14)\bigr);
\end{multline*}
to take into account the range of \(i_{234}^{1234}\) as well, we need an exact sequence containing \(\delta_{1234}^{14}\circ i_{234}^{1234}\), which is the generator of \(\Nattrafo_1(234,14)\cong\Z\).  Since there is no such exact sequence, our method breaks down at this point.

Another symptom but not a cause of problems is that the map \(\delta_{124}^4\) that describes \(M'(124)\) is not the longest map out of \(124\): that would be \(\delta_{124}^{34}\).

As we shall see, the analogues of Theorems \ref{the:projective_module_Nattrafo} and~\ref{the:length_one_example} become false for the space~\(X\).  First, there is a non-projective exact module~\(M\) with free~\(M_\sesi\); secondly, there is a module that has no projective resolution of length~\(1\); thirdly, there are \(A,B\in\Bootstrap(X)\) with \(\Ideal^2(A,B)\neq0\).  Hence Theorem~\ref{the:long-resolution_problem} provides non-isomorphic objects in the bootstrap class \(\Bootstrap(X)\) with isomorphic filtrated \(\K\)\nb-theory.  The construction of these counterexamples follows the above pattern: first we find a counterexample to Theorem~\ref{the:projective_module_Nattrafo}, which we use to find one for Theorem~\ref{the:length_one_example}, which is then used to find an example as in Theorem~\ref{the:long-resolution_problem}.

We begin with the unexpected non-projective module.  Let~\(P_Y\) for \(Y\in\Loclo(X)^*\) denote the free \(\Nattrafo\)\nb-module on~\(Y\), that is,
\[
P_Y(Z) = \Nattrafo_*(Y,Z),
\qquad
\Hom_\Nattrafo(P_Y, N) \cong N(Y)
\]
for any \(Y,Z\in\Loclo(X)^*\) and any \(\Nattrafo\)\nb-module~\(N\).  A natural transformation \(\FK_Y\Rightarrow\FK_Z\) corresponds to an element in \(\Nattrafo_*(Y,Z) \cong P_Y(Z) \cong \Hom_\Nattrafo(P_z,P_Y)\) and thus induces a module homomorphism \(P_Z\to P_Y\) in the opposite direction.  Hence the three arrows \(124,134,234\to1234\) in~\eqref{eq:Auslander-Reiten} induce a module homomorphism
\[
j\colon P_{1234} \to
P^0 \defeq P_{124} \oplus P_{134} \oplus P_{234}.
\]
Table~\ref{tab:nattrafo_counterexample} shows that there are no module homomorphisms \(P^0\to P_{1234}\), that is, no non-zero natural transformations from \(1234\) to \(124\), \(134\), or \(234\).

The crucial observation is that~\(j\) is a monomorphism, so that \(P_{1234}\) becomes a submodule of~\(P^0\).  Since the longest natural transformations out of \(1234\) are those to \(14\), \(24\) and~\(34\), this follows from the elementary observations that the maps
\[
\Nattrafo_*(1234,j4) \to
\Nattrafo_*(1234\setminus j,j4)
\]
for \(j=1,2,3\) are, respectively, the identity map on~\(\Z\).  This follows from the exactness of free modules because \(\Nattrafo_*(j,j4)=0\) by Table~\ref{tab:nattrafo_counterexample}.

We describe the quotient
\[
M \defeq P^0/j(P_{1234})
\]
by its values \(M(Y)\) for \(Y\in\Loclo(X)^*\) as in~\eqref{eq:Auslander-Reiten}:
\begin{equation}
  \label{eq:non-projective_module}
  \begin{gathered}
    \xymatrix@C+1em{
      & 0 \ar[r]^{i} \ar[dr]^{i}&
      \Z \ar[dr]^{i}&& \Z \ar@{=}[dr]|-\circ^{\delta}\\
      \Z[1] \ar[ur]^{i} \ar[r]^{i} \ar[dr]^{i}&
      0 \ar[ur]^{i} \ar[dr]^{i}&
      \Z \ar[r]^{i}&
      \Z^2 \ar[ur]^{r} \ar[r]^{r} \ar[dr]^{r}&
      \Z\ar@{=}[r]|-\circ^{\delta}& \Z[1]\\
      & 0 \ar[r]^{i} \ar[ur]^{i}&
      \Z \ar[ur]^{i}&& \Z \ar@{=}[ur]|-\circ^{\delta}
    }
  \end{gathered}
\end{equation}
The boundary maps~\(\delta\) act by isomorphisms on~\(M\) because \(M(j4)=0\) for \(j=1,2,3\).  The other maps can be understood by writing \(M(1234)=\Z^3/\langle(1,1,1)\rangle\) and \(M(j)=\Z^2/\langle(1,1)\rangle\) for \(j=1,2,3\) as quotients.  The three maps \(\Z\to \Z^2\) correspond to the three coordinate embeddings \(\Z\into\Z^3\), the maps \(\Z^2\to\Z\) to the projections \(\Z^3\prto\Z^2\) onto coordinate hyperplanes.

The projective resolution
\begin{equation}
  \label{eq:resolve_M}
  0 \to P_{1234} \xrightarrow{j} P^0 \prto M
\end{equation}
does not split because there exist no non-zero morphisms \(P^0\to P_{1234}\).  Hence~\(M\) is not projective.  But~\(M_\sesi\) is free, and~\(M\) is exact because the exact modules form an exact category and \(P_{1234}\) and \(P^0\) are exact.  Thus~\(M\) is a counterexample to Theorem~\ref{the:projective_module_Nattrafo}.

The module~\(M\) is directly related to the problem with describing \(\Nattrafo_\nil\cdot M(1234)\) encountered above.  Since \(\Hom_\Nattrafo(P_Y,N) \cong N(Y)\) for any \(\Nattrafo\)\nb-module~\(N\) and any \(Y\in\Loclo(X)^*\), the resolution~\eqref{eq:resolve_M} provides an exact sequence
\begin{multline*}
  0 \to \Hom_\Nattrafo(M,N)
  \\\to N(124)\oplus N(134) \oplus N(234)
  \to N(1234) \to \Ext^1_\Nattrafo(M,N) \to 0,
\end{multline*}
so that
\[
\Ext^1_\Nattrafo(M,N) \cong N(1234)/\Nattrafo_\nil\cdot N(1234)
\cong N_\sesi(1234).
\]

Now we use~\(M\) to construct a counterexample for Theorem~\ref{the:length_one_example}.  Let \(k\in\N_{\ge2}\) and let \(M_k\defeq M/k\cdot M\); that is, we replace~\(\Z\) by \(\Z/k\) everywhere in~\eqref{eq:non-projective_module}.  This module has a projective resolution of length~\(2\) of the form
\begin{equation}
  \label{eq:resolve_counterexample}
  0 \to
  P_{1234} \xrightarrow{(-k,j)} P_{1234} \oplus P^0
  \xrightarrow{(j,k)} P^0 \prto M_k,
\end{equation}
where~\(k\) denotes multiplication by~\(k\).  Using this resolution, we compute
\[
\Ext^2(M_k,P_{1234}) \cong \Z/k,
\qquad
\Ext^1(M_k,P_{1234}) \cong \Hom(M_k,P_{1234}) \cong 0
\]
because there are no no-zero morphisms \(P^0\to P_{1234}\).  Of course, the generator of \(\Ext^2(M_k,P_{1234})\) is the class of the projective resolution~\eqref{eq:resolve_counterexample}.  Hence~\(M_k\) admits no projective resolution of length~\(1\) and is a counterexample to Theorem~\ref{the:length_one_example}.

Now we claim that~\(M_k\) is the filtrated \(\K\)\nb-theory of some \(\Cst\)\nb-algebra~\(A_k\) over~\(X\) in the bootstrap class \(\Bootstrap(X)\).  To begin with, \(M\) is the filtrated \(\K\)\nb-theory of some such \(\Cst\)\nb-algebra~\(A\) by Theorem~\ref{the:FK_exhausts_conditional}.  Let~\(B_k\) be a \(\Cst\)\nb-algebra in the bootstrap class with \(\K_0(B_k) = \Z/k\) and \(\K_1(B_k)=0\); for instance, \(B_k\) could be the Cuntz algebra~\(\mathcal{O}_{k+1}\).  Then \(A_k\defeq A\otimes B_k\) has filtrated \(\K\)\nb-theory~\(M_k\) by the K\"unneth Theorem for the \(\K\)\nb-theory of tensor products.

\begin{theorem}
  \label{the:counterexample_special_X}
  Let~\(A_k\) be a \(\Cst\)\nb-algebra in the bootstrap class with \(\FK(A_k)\cong M_k\) as constructed above.  Then~\(A_k\) is not \(\Ideal^2\)\nb-projective.  Hence there exist \(B,D\in\Bootstrap(X)\) that are not \(\KK(X)\)\nb-equivalent but with the same filtrated \(\K\)\nb-theory.
\end{theorem}

\begin{proof}
  The second assertion follows from the first one using Theorem~\ref{the:long-resolution_problem} applied to the bootstrap class \(\Bootstrap(X)\) and the restriction of~\(\Ideal\) to \(\Bootstrap(X)\).

  It remains to prove that~\(A_k\) cannot be \(\Ideal^2\)\nb-projective.  To see this, we lift the resolution~\eqref{eq:resolve_counterexample} to an \(\Ideal\)\nb-projective resolution
  \[
  \xymatrix@1{
  0 \ar[r]|-\circ& P_2 \ar[r]|-\circ& P_1 \ar[r]|-\circ& P_0 \ar[r]& A_k}
  \]
  in \(\Bootstrap(X)\) with boundary maps of degree~\(1\), and embed the latter in a phantom tower (see~\cite{Meyer:Homology_in_KK_II}):
  \[
  \xymatrix@C-1.5em{
    A_k\ar@{=}[r]&
    N_0\ar[rr]^{\iota_0^1}&&
    N_1\ar[rr]^{\iota_1^2}\ar[dl]|-\circ&&
    N_2\ar[rr]^{\iota_2^3}\ar[dl]|-\circ&&
    N_3\ar@{=}[rr]\ar[dl]|-\circ&&
    N_3\ar@{=}[rr]\ar[dl]|-\circ&&
    \dotsb\\
    &&P_0\ar[ul]^{\pi_0}&&
    P_1\ar[ul]^{\pi_1}\ar[ll]&&
    P_2\ar[ul]^{\pi_2}\ar[ll]&&
    0\ar[ul]\ar[ll]&&
    \dotsb\ar[ll]
  }
  \]

  The inductive system \((N_j,\iota_j^{j+1})\) becomes constant at~\(N_3\) because \(P_j=0\) for \(j\ge 3\).  Since~\(A_k\) belongs to the bootstrap class, \(N_3\cong0\) (see the proof of \cite{Meyer:Homology_in_KK_II}*{Proposition 4.5}).  This implies \(N_2\cong P_2\).

  The composite map \(\iota_0^2\colon A_k = N_0\to N_2 \cong P_2\) belongs to~\(\Ideal^2\).  Suppose that~\(A_k\) were \(\Ideal^2\)\nb-projective.  Then \(\iota_0^2=\iota_1^2\circ \iota_0^1\) would vanish, and the long exact homology sequence would yield that the map \(\iota_1^2\colon N_1\to N_2\) must factor through the map \(N_1\to P_0\).  But
  \[
  \KK_*(X;P_0,P_2) \cong
  \Hom_\Nattrafo\bigl(\FK(P_0),\FK(P_2)\bigr)
  = \Hom_\Nattrafo(P^0,P_{1234})
  = 0.
  \]
  Here we have used that filtrated \(\K\)\nb-theory, by universality, is fully faithful on \(\Ideal\)\nb-projective objects and that there are no non-zero module homomorphisms \(P^0\to P_{1234}\).  Since~\(\iota_1^2\) factors through the zero group, it must be the zero map.  But then the map \(P_1\to N_1\) must be a split surjection, so that~\(N_1\) is \(\Ideal\)\nb-projective.  Then the \(\Ideal\)\nb-exact triangle \(\Sigma A_k \to \Sigma N_1\to P_0 \to A_k\) provides an \(\Ideal\)\nb-projective resolution of~\(A_k\) of length~\(1\), which is impossible because \(\FK(A_k)\cong M_k\) has no projective resolution of length~\(1\).  As a consequence, \(A_k\) is not \(\Ideal^2\)\nb-projective.
\end{proof}

We can make the two non-equivalent \(\Cst\)\nb-algebras over~\(X\) with the same filtrated \(\K\)\nb-theory more explicit.  One of them is \(A_k\oplus \Sigma\Repr_{1234}\), the other one is the mapping cone of the map \(\iota_0^2\colon A_k=N_0\to N_2\cong \Repr_{1234}\) in the phantom tower above.  Both have \(M_k\oplus P_{1234}[1]\) as their filtrated \(\K\)\nb-theory.

This counterexample shows that filtrated \(\K\)\nb-theory does not yet classify purely infinite stable nuclear separable \(\Cst\)\nb-algebras in the bootstrap class.

\begin{remark}
  \label{rem:refine_coefficient}
  Refining filtrated \(\K\)\nb-theory by taking filtrated \(\K\)\nb-theory with coefficients does not help.  This gets rid of the counterexample~\(A_k\) constructed above, but other objects of \(\Bootstrap(X)\) without projective resolution of length~\(1\) remain.  An example is \(A\otimes B\), where~\(B\) is a \(\Cst\)\nb-algebra in the bootstrap class with \(\K_*(B)=\mathbb{Q}[0]\) such as an appropriate UHF-algebra.  Its filtrated \(\K\)\nb-theory is \(M\otimes\mathbb{Q}\).  This also has cohomological dimension~\(2\), and this is not affected much by taking \(\K\)\nb-theory with coefficients because \(M\otimes\mathbb{Q}\) is torsion-free.
\end{remark}

\subsection{A refined invariant}
\label{sec:refined_invariant}

There are at least two ways to identify the source of the problem for the space~\(X\).  The first point of view is that what is missing is an exact sequence that has the generator~\(\alpha\) of \(\Nattrafo_1(234,14)\) as its connecting map.  The map~\(\alpha\) corresponds to a map \(\Sigma\Repr_{14}\to \Repr_{234}\) between the representing objects, which we also denote by~\(\alpha\).  In the triangulated category \(\KKcat(X)\), we can embed the latter map in an exact triangle
\begin{equation}
  \label{eq:new_exact_triangle}
  \Sigma \Repr_{14} \xrightarrow{\alpha}
  \Repr_{234} \to \Repr_{12344} \to \Repr_{14}.
\end{equation}
The notation~\(\Repr_{12344}\) will be explained later.  The functors these objects represent sit in a long exact sequence
\begin{equation}
  \label{eq:exact_12344}
  \dotsb \to \FK_{14} \to \FK_{12344} \to \FK_{234}
  \xrightarrow{\alpha} \FK_{14}[1] \to \dotsb
\end{equation}
which is precisely what we want.  The second point of view is that the troublemaker is the non-projective module~\(M\).  Since~\(M\) has a projective resolution of length~\(1\), there is a \emph{unique} object in the bootstrap class with filtrated \(\K\)\nb-theory~\(M\).  Actually, this yields the same object as the first point of view:

\begin{lemma}
  \label{lem:Repr_12344}
  The non-projective module~\(M\) above agrees with \(\FK(\Repr_{12344})\).
\end{lemma}

\begin{proof}
  The map \(\FK_Y(\alpha)\) vanishes for almost all \(Y\in\Loclo(X)^*\) simply because the graded groups involved have different parity or one of them vanishes.  The only exception is \(Y=14\).  The group \(\FK_{14}(\Repr_{14})=\Nattrafo(14,14)\) is generated by the identity natural transformation.  Since~\(\alpha\) is the generator of \(\Nattrafo_1(234,14)\), the map \(\FK_{14}(\alpha)\) is invertible.

  Now we apply \(\FK\) to the long exact sequence for the given exact triangle.  Since \(\FK(\alpha)\) vanishes on most~\(Y\) and is invertible for \(Y=14\), we can easily compute the groups \(\FK_Y(\Repr_{12344})\).  We get the same groups as for the module~\(M\).  It remains to check that the isomorphism can be chosen as an \(\Nattrafo\)\nb-module homomorphism.  The main step is to check that the map
  \[
  \Z^2 \cong \FK_{124}(\Repr_{12344}) \oplus
  \FK_{134}(\Repr_{12344})
  \to \FK_{1234}(\Repr_{12344}) \cong \Z^2
  \]
  is invertible.  Together with the known relations between the various natural transformations, this implies the assertion.  We omit the details of this computation.
\end{proof}

The representing object~\(\Repr_{12344}\) is an algebra of functions on a two-dimensional simplicial complex, which we do not describe here because it is not illuminating.  The functor that it represents, however, can be described rather nicely as follows.  Let~\(A\) be a \(\Cst\)\nb-algebra over~\(X\).  Pull back the extension \(A(14)\into A(124) \prto A(2)\) along the quotient map \(A(234)\prto A(2)\) to an extension \(A(14) \into A(12344) \prto A(234)\).  The object \(\Repr_{12344}\) represents the functor
\begin{equation}
  \label{eq:new_nice}
  \KK_*(X;\Repr_{12344},A) \cong \K_*\bigl(A(12344)\bigr).
\end{equation}
To see this, two observations are necessary.  First, \(\K_*\bigl(\Repr_{12344}(12344)\bigr)\cong\Z\); the generator of this group yields a natural transformation between the two functors in~\eqref{eq:new_nice}.  Secondly, this natural transformation is invertible.  This follows from the Five Lemma, once we know that it extends the known natural isomorphisms
\[
\KK_*(X;\Repr_Y,A) \cong \K_*\bigl(A(Y)\bigr)
\]
for \(Y=14\) and \(Y=234\) to a chain map between the long exact sequences that we get from~\eqref{eq:new_exact_triangle} and from the extension \(A(14)\into A(12344)\prto A(234)\).  This extension also explains the notation \(\Repr_{12344}\).

Now we augment filtrated \(\K\)\nb-theory by adding the covariant functor
\[
B\mapsto \FK_{12344}(B)
\defeq \K_*\bigl(A(12344)\bigr)
\cong \KK_*(X;\Repr_{12344},B).
\]
The new invariant takes values in the category of countable \(\Nattrafo'\)\nb-modules, where \(\Nattrafo'\) is the \(\Z/2\)\nb-graded category whose object set is \(\Loclo'\defeq \Loclo(X)^*\sqcup \{12344\}\) and whose morphisms are the natural transformations between the various filtrated \(\K\)\nb-groups, including now also \(\FK_{12344}\).  These natural transformations can be computed by the Yoneda Lemma:
\[
\Nattrafo'_*(Y,Z)
\cong \KK_*(X;\Repr_Z,\Repr_Y)
\cong \FK_Z(\Repr_Y)
\]
holds for all \(Y,Z\in\Loclo'\).  The category ring for \(\Nattrafo'_*\) is simply the ring \(\KK_*(X;R,R)\) where
\[
R\defeq \bigoplus_{Y\in\Loclo'} \Repr_Y.
\]

We replace the ideal~\(\Ideal\) in \(\KKcat(X)\) studied above by the kernel~\(\Ideal'\) of the enriched filtrated \(\K\)\nb-theory functor
\[
\FK'\colon \KKcat(X) \to \CMod{\Nattrafo'}.
\]
The same arguments as above show that there are enough \(\Ideal'\)\nb-projective objects and that \(\FK'\) is the universal \(\Ideal'\)\nb-exact stable homological functor.

The passage from~\(\Ideal\) to~\(\Ideal'\) has improved the situation because \(\Repr_{12344}\) has now been promoted to an \(\Ideal'\)\nb-projective object and, therefore, ceases to cause trouble.  In principle, something similar can be done in great generality: whenever we have an object of the Abelian approximation that has a projective resolution of length~\(1\), we can lift it uniquely to an object of the triangulated category and refine the ideal by intersecting it with the kernel of the functor this lifted object represents.  However, the policy to quieten troublemakers by promotion has the tendency to encourage new troublemakers, so that it is not clear whether this general strategy always resolves all problems after finitely many steps.  But in the relatively simple example at hand, this turns out to be the case.

To check this, we must describe the category \(\Nattrafo'\).  If \(Y,Z\in\Loclo(X)^*\), then \(\Nattrafo'_*(Y,Z) = \Nattrafo_*(Y,Z)\) is given by the table on page~\pageref{tab:nattrafo_counterexample}.  Furthermore, if \(Z\in\Loclo(X)^*\), then \(\Nattrafo'_*(12344,Z) \cong \FK_Z(\Repr_{12344}) = M(Z)\) by Lemma~\ref{lem:Repr_12344}, and this is described in~\eqref{eq:non-projective_module}.  The upshot is:
\begin{itemize}
\item there are even natural transformations from \(\FK_{12344}\) to \(\FK_{124}\), \(\FK_{134}\), \(\FK_{234}\)---the generators of the respective groups of natural transformations---such that any natural transformation \(\FK_{12344}\Rightarrow\FK_Z\) with \(Z\in\Loclo(X)^*\) is a sum of natural transformations that factor through one of these three and a natural transformation \(\FK_{ij4}\Rightarrow\FK_Z\);

\item the sum of the three natural transformations \(\FK_{12344}\Rightarrow\FK_{1234}\) via \(\FK_{124}\), \(\FK_{134}\) and \(\FK_{234}\) vanishes, and all other relations follow from these and the already known ones listed after~\eqref{eq:Auslander-Reiten}.
\end{itemize}

The exact triangle~\eqref{eq:new_exact_triangle} yields a long exact sequence
\[
\dotsb \to
\Nattrafo'_{*+1}(Y,234) \xrightarrow{\alpha}
\Nattrafo'_*(Y,14) \to
\Nattrafo'_*(Y,12344) \to
\Nattrafo'_*(Y,234) \to \dotsb,
\]
which we may use to compute \(\Nattrafo'_*(Y,12344)\) for all \(Y\in\Loclo'\).  The map~\(\alpha\) induces an isomorphism for \(Y=234\) and the zero map for all other~\(Y\) because the source and target have opposite parity or one of them vanishes.  Thus
\[
\begin{array}{l*{6}{|c}}
  Y&4&14,24,34&124,134,234&1234&1,2,3&12344\\\hline
  \Nattrafo'_*(Y,12344)&\Z^2&\Z&0&\Z[1]&\Z[1]&\Z\\
\end{array}
\]
These groups inherit from~\(M\) their invariance under permutations of \(1,2,3\).  Inspecting composition with natural transformations in~\(\Nattrafo\), we arrive at the following:
\begin{itemize}
\item there are even natural transformations \(\FK_{j4}\Rightarrow\FK_{12344}\) for \(j=1,2,3\), such that any natural transformation \(\FK_Y\Rightarrow\FK_{12344}\) with \(Y\in\Loclo(X)^*\) factors through one of them;

\item the sum of the three natural transformations \(\FK_4\Rightarrow\FK_{12344}\) vanishes,

\item the natural transformations \(\FK_{j4}\Rightarrow\FK_{1234\setminus j}\) via \(\FK_{12344}\) vanish;

\item all other relations follow from these and the already known ones.
\end{itemize}
As one may expect, the basic natural transformations \(\FK_{14}\Rightarrow\FK_{12344}\Rightarrow\FK_{234}\) are induced by the maps \(\Repr_{234} \to \Repr_{12344} \to \Repr_{14}\) in the exact triangle~\eqref{eq:new_exact_triangle}.

The indecomposable morphisms of the new category \(\Nattrafo'\) are the maps in the following diagram:
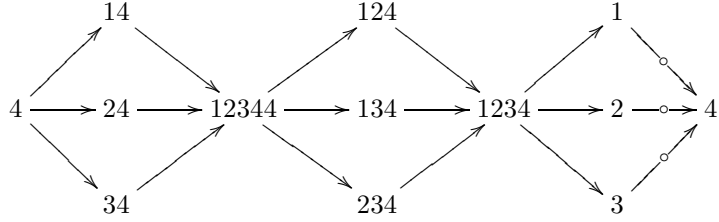
\[
\xymatrix{
  & 14 \ar[dr]&&
  124 \ar[dr]&& 1
  \ar[dr]|-\circ\\
  4 \ar[ur] \ar[r] \ar[dr]&
  24 \ar[r]&12344\ar[r]\ar[ur]\ar[dr]&
  134 \ar[r]&
  1234 \ar[ur] \ar[r] \ar[dr]&
  2\ar[r]|-\circ&4\\
  & 34 \ar[ur]&&
  234 \ar[ur]&& 3
  \ar[ur]|-\circ
}
\]

The category ring of \(\Nattrafo'\) again has the by now familiar structure: it is a split nilpotent extension of the semisimple algebra \(\Nattrafo'_\sesi \cong \Z^{\Loclo'}\) spanned by the identity transformations on the objects and a nilpotent ideal \(\Nattrafo'_\nil\) that is the subgroup generated by \(\Nattrafo'(Y,Z)\) with \(Y\neq Z\).

\begin{definition}
  \label{def:exact_over_extended}
  A module over \(\Nattrafo'\) is \emph{exact} if it is exact as an \(\Nattrafo\)\nb-module and the three sequences
  \[
  \dotsb \to N_{*+1}(ij4) \to N_*(k4) \to N_*(12344) \to
  N_*(ij4) \to \dotsb
  \]
  for \(\{i,j,k\}=\{1,2,3\}\) are exact as well.
\end{definition}

The range of the invariant \(\FK'\) consists of exact \(\Nattrafo'\)\nb-modules; the three new exact sequences are, in fact, equivalent for symmetry reasons, and the extension
\[
\dotsb \to N_{*+1}(234) \to N_*(14) \to N_*(12344) \to
N_*(234) \to \dotsb
\]
is built into the definition of~\(\FK_{12344}\).

Let~\(N\) be an exact \(\Nattrafo'\)\nb-module and let \(N'\defeq \Nattrafo'_\nil\cdot N\).  The description of \(N'(14)\), \(N'(1)\), and \(N'(4)\) is the same as for the category \(\Nattrafo\), so that these groups remain kernels of certain maps, as needed.  Furthermore, \(N'(1234)\) is the kernel of the map \(N(1234)\to N(12344)[1]\) induced by the generator of \(\Nattrafo_1(1234,12344)\), so that the problem that appeared for the category \(\Nattrafo\) is cured.

The computation of \(N'(124)\) changes because this group is now the range of the arrow \(N(12344)\to N(124)\).  But this is part of a long exact sequence because~\(N\) is exact, and we get
\[
N'(124) = \ker \bigl(N(124) \to N(34)[1]\bigr),
\]
and similarly for \(N'(134)\) and \(N'(234)\).

Finally, \(N'(12344)\) is the sum of the ranges of the maps \(N(j4)\to N(12344)\) for \(j=1,2,3\).  Using exactness, we identify this in two steps with the kernel of the map \(N(12344)\to N(4)[1]\) induced by the generator of \(\Nattrafo'_1(12344,4)\).

As a result, the submodule \(\Nattrafo'_\nil\cdot N\) is described using kernels of maps \(N(Y)\to N(Z)\).  By the way, these arrows are the longest arrows starting at~\(Y\) as in Remark~\ref{rem:longest}.  The same arguments as for totally ordered spaces now show:

\begin{theorem}
  \label{the:counterexample_cured_projective}
  An \(\Nattrafo'\)-module~\(N\) is free if and only if it is projective, if and only if it is exact and \(N(Y)\) is a free group for all \(Y\in\Loclo'\).
\end{theorem}

\begin{theorem}
  \label{the:counterexample_cured_resolution}
  An \(\Nattrafo'\)-module~\(N\) has a projective resolution of length~\(1\) if and only if it is exact.
\end{theorem}

\begin{theorem}
  \label{the:counterexample_cured_UCT}
  Let \(A\) and~\(B\) be \(\Cst\)\nb-algebras over the four-point space~\(X\) under consideration.  If~\(A\) belongs to the bootstrap class \(\Bootstrap(X)\), then there is a natural short exact sequence
  \[
  \Ext^1_{\Nattrafo'}\bigl(\FK'(A)[1],\FK'(B)\bigr)
  \into \KK_*(X;A,B)
  \prto \Hom_{\Nattrafo'}\bigl(\FK'(A),\FK'(B)\bigr).
  \]
  In particular, morphisms \(\FK'(A)\to\FK'(B)\) lift to elements in \(\KK_*(X;A,B)\).  If both \(A\) and~\(B\) belong to the bootstrap class, then an isomorphism \(\FK'(A)\cong\FK'(B)\) lifts to a \(\KK(X)\)\nb-equivalence.
\end{theorem}

\begin{corollary}
  \label{cor:counterexample_cured_classify}
  The map \(A\mapsto \FK'(A)\) is a bijection between the set of isomorphism classes of tight, stable, purely infinite, separable, nuclear \(\Cst\)\nb-algebras over~\(X\) with simple subquotients in the bootstrap class and the set of isomorphism classes of countable exact \(\Nattrafo'\)\nb-modules.
\end{corollary}

\section{Conclusion}
\label{sec:conclusion}

We have obtained a Universal Coefficient Theorem that computes \(\KK_*(X;A,B)\) for \(A\) in the bootstrap class and~\(X\) of a very special form, namely, \(\{1,\dotsc,n\}\) with the Alexandrov topology from the total order.  This Universal Coefficient Theorem can be used to carry over classification results for simple, nuclear, purely infinite \(\Cst\)\nb-algebras to nuclear, purely infinite \(\Cst\)\nb-algebras with primitive ideal space~\(X\), using filtrated \(\K\)\nb-theory as the invariant.

For general finite topological spaces~\(X\), we still get a spectral sequence that computes \(\KK_*(X;A,B)\) using filtrated \(\K\)\nb-theory, but this spectral sequence need not degenerate to an exact sequence, so that isomorphisms on filtrated \(\K\)\nb-theory need not lift to \(X\)\nb-equivariant \(\KK\)\nb-equivalences.  In fact, we have found a counterexample.  At the same time, we were able to fix the counterexample by refining filtrated \(\K\)\nb-theory.  It is unclear whether such a refinement is available for all finite topological spaces and how it looks like.

\end{document}